\documentclass[11pt]{article}

\usepackage[matrix,arrow,arc]{xy}
\usepackage{enumerate}
\usepackage{amssymb}

\newcommand{\nd}{\noindent}
\newcommand{\vu}{\vspace{.1cm}}
\newcommand{\vd}{\vspace{.2cm}}
\newcommand{\vt}{\vspace{.3cm}}
\newcommand{\vq}{\vspace{.4cm}}

\newcommand{\qed}{\nolinebreak\hfill$\Box$\par\medbreak}

\newcommand{\Cc}{{\mathbb C}}

\newcommand{\A}{\mathcal A}
\newcommand{\B}{\mathcal B}

\newcommand{\semi}{\mathcal{S}}

\newcommand{\p}{\cdot}

\newcommand{\cosemi}{>\!\!\blacktriangleleft}

\newcommand{\ka}{{\kappa}}
\newcommand{\gm}{\gamma}

\newcommand{\benu}{\begin{enumerate}}
\newcommand{\enu}{\end{enumerate}}

\newcommand{\beqna}{\begin{eqnarray}}
\newcommand{\eqna}{\end{eqnarray}}
\newcommand{\beqnast}{\begin{eqnarray*}}
\newcommand{\eqnast}{\end{eqnarray*}}
\newcommand{\beqn}{\begin{equation}}
\newcommand{\eqn}{\end{equation}}
\newcommand{\beqnst}{\begin{equation*}}
\newcommand{\eqnst}{\end{equation*}}

\newcommand{\uinv}{u^{-1}}
\newcommand{\um }{\mathbf{1}_A}
\newcommand{\AH}{A\#_{(\alpha,\omega)}H}
\newcommand{\Hom}{\operatorname{Hom}}
\newcommand{\gi}{\gamma^\prime}
\newcommand{\gb}{\overline{\gamma}}
\newcommand{\ot}{\otimes}

\newcommand{\m}{{}^{-1}}

\usepackage{amsmath,amsthm,amssymb}
\usepackage{amsxtra,amscd, amsbsy}
\usepackage{eucal}

\newtheorem{lemma}{Lemma}[section]
\newtheorem{prop}{Proposition}[section]

\newtheorem{thm}{Theorem}[section]
\newtheorem{defi}{Definition}[section]

\theoremstyle{definition}
\newtheorem{ex}{Example}[section]
\newtheorem{remark}{Remark}[section]

\newcommand{\vai}{\rightarrow}

\usepackage{enumerate}

\begin{document}

\begin{center}
\Large{\bf Twisted partial actions of Hopf algebras}\footnote[0]{The first and second authors were partially supported by Funda\c{c}\~ao Arauc\'aria of Brazil, 490/16032. The third and the fourth authors were partially supported by Fapesp of Brazil. The first and third author were also partially supported by CNPq of Brazil.}
\end{center}

\begin{center}
{\bf Marcelo Muniz S. Alves$^1$, Eliezer Batista$^2$,\\ Michael Dokuchaev$^3$ and Antonio Paques$^4$}\\

\vt

{ \footnotesize

$^1$ Departamento de Matem\'atica\\
Universidade Federal do Paran\'a\\
81531-980, Curitiba, PR, Brazil\\
E-mail: {\it marcelomsa@ufpr.br}\\

$^2$ Departamento de Matem\'atica\\
Universidade Federal de Santa Catarina\\
88040-900, Florian\'opolis, SC, Brazil\\
E-mail: {\it ebatista@mtm.ufsc.br}\\

$^3$ Instituto de Matem\'atica e Estat\'istica\\
Universidade de S\~ao Paulo\\
05508-090, S\~ao Paulo, SP, Brazil \\
E-mail: {\it dokucha@ime.usp.br}\\

$^4$Instituto de Matem\'atica\\
Universidade Federal do Rio Grande do Sul\\
91509-900, Porto Alegre, RS, Brazil\\
E-mail: {\it paques@mat.ufrgs.br}}
\end{center}

\begin{abstract} In this work, the notion of a twisted partial Hopf action is introduced as a unified approach for twisted partial group actions, partial Hopf actions and twisted actions of Hopf algebras. The conditions on partial cocycles are established in order to construct partial crossed products, which are also related to partially cleft extensions of algebras.  Examples are elaborated using algebraic groups.
\end{abstract}

\section{Introduction}

The desire to endow important  classes of $C^*$-algebras generated by partial isometries with a structure of a more general crossed product led to the concept of a partial  group action, introduced in  \cite{E-1}, \cite{Mc}, \cite{E0}, \cite{E1}. The new structure permitted to obtain relevant  results on $K$-theory,   ideal structure and representations  of the algebras under consideration, as well as to treat amenability questions, especially amenability of $C^*$-algebraic bundles (also called Fell bundles),  using both partial actions and the related concept of a partial representation.  Amongst prominent classes of $C^*$-algebras endowed with the structure of non-trivial  crossed products by partial actions one may list the  Bunce-Deddens and the   Bunce-Deddens-Toeplitz  algebras \cite{E-2}, the approximately finite dimensional algebras \cite{E-3}, the Toeplitz algebras of quasi-ordered groups, as well as the  Cuntz-Krieger algebras \cite{ELQ}, \cite{QR}.\\

 The algebraic study of partial actions and partial representations was initiated in \cite{E1}, \cite{DEP} and \cite{DE}, motivating investigations in diverse directions.  In particular, the Galois theory of partial group actions developed in \cite{DFP} inspired  further Galois theoretic results in \cite{CaenDGr}, as well as the  introduction and study of partial Hopf actions and coactions in \cite{CJ}. The latter paper became in turn the starting point for further investigation of partial Hopf (co)actions in \cite{AB}, \cite{AB2}  and \cite{AB3}. The Galois theoretic treatment in \cite{CaenDGr} was based on a coring ${\mathcal C} $ constructed for an idempotent partial action of a finite group. The coring ${\mathcal C} $ was shown to fit the general theory of cleft bicomodules  in \cite{bohmverc}, and, in addition, in \cite{Brz}  descent theory for corings was  applied, using   ${\mathcal C} ,$ to define  non-Abelian Galois cohomology  ($i=0, 1$) for idempotent partial Galois actions of finite groups.\\

The general notion of a (continuous) twisted partial action of a locally compact group on a
$C^*$-algebra (a twisted partial $C^*$-dynamical system) and the cor\-res\-pon\-ding crossed pro\-ducts were given by R. Exel in \cite{E0}.  The new construction permitted to show that any se\-cond countable $C^*$-algebraic bundle, which satisfies a certain regularity condition (automatically verified  if the unit fiber algebra is stable), is a $C^*$-crossed product of the unit fiber algebra by a continuous partial action of the base group. The algebraic version of the latter fact was established in \cite{DES1}. The
importance of partial actions and  partial representations  was reinforced
by R. Exel in \cite{E3} where, among other results, it was proved that given a field $K$
of characteristic $0,$ a group $G$ and subgroups $H, N \subseteq G$ with $N$  normal in
$G$ and $H$  normal in $N,$ there is a twisted partial action $\theta $ of $G/N$ on the group
algebra $K(N/H)$ such that the Hecke algebra ${\mathcal H}(G,H)$ is isomorphic to the
crossed product $K(N/H) \ast _{\theta} G/N.$ More recent algebraic results on twisted partial
actions and corresponding crossed products were obtained in \cite{BLP}, \cite{DES2}  and \cite{PSantA}. The algebraic concept of twisted partial actions also motivated the study of projective partial group representations, the corresponding partial Schur Multiplier and the relation  to partial group actions with $K$-valued twistings in \cite{DN} and \cite{DN2}, contributing towards the elaboration of a background for a general cohomological theory based on partial actions. Further information around partial actions may be consulted in the survey \cite{D}.\\

The aim of this article is to introduce and study twisted partial Hopf actions on rings. The general definitions are given  in Section 2, including that of a  partial crossed product. The cocycle and normalization conditions are needed in order to make the partial crossed product to be both associative and unital. As expected, restrictions of usual (global) twisted Hopf actions naturally result in twisted partial Hopf actions. Idempotent twisted partial actions of groups  give natural examples of twisted partial actions of Hopf group algebras.  Less evident examples may be obtained using algebraic groups, as it is shown in  Section 3.   Actions of an affine algebraic group on affine varieties give rise to  coactions of the corresponding commutative Hopf algebra $H$  on  the coordinate algebras of the varieties,  restrictions of which produce  concrete examples of partial Hopf coactions.  Then one  may  dualize in order to obtain  partial Hopf actions.  This works theoretically, but the elaboration of a concrete example needs some work. One possibility is to try to identify the finite dual $H^0$ for a specific $H$ obtained this way.  A more flexible possibility is to find a concrete Hopf algebra $H_1$ such that $H$ and $H_1$ form a dual pairing. Then Proposition 8 from \cite{AB} produces a partial action of $H_1.$ One still wishes to transform it into a twisted one, which in the setting specified in Section 3 is not difficult.   A concrete example is elaborated in Proposition \ref{Ex:AlgGrHopf}.\\

In order to treat the convolution invertibility of the partial cocycle in a manageable way, we introduce symmetric twisted partial Hopf actions in Section 4 and establish some useful technical formulas.  Our definition is inspired by the case of twisted partial group actions. We also show that a restriction of a global twisted Hopf action with convolution invertible cocycle gives a symmetric twisted partial  Hopf action. Theorem \ref{the41} relates isomorphisms of crossed products by symmetric twisted partial actions with a kind of ``partial coboundaries,'' establishing an analogue of a corresponding result known in the global case.\\

The last Section 5 is dedicated to the notion of partial cleft extensions and its relation with partial crossed products, in a quite similar fashion as it is done in classical Hopf algebra theory.  The definition of a partial cleft extension reflects the ``partiality'' in more than one ways, incorporating, in particular, some equalities already proved to be significant in the study of partial group actions and partial representations (see Remark \ref{interaction}).  
Then the main result Theorem \ref{the51} states that the partial cleft extensions over the coinvariants $A$ are exactly the crossed products by symmetric twisted partial Hopf actions on $A.$  

\section{Twisted partial actions and partial crossed
products}

In  this paper, except Section~\ref{ExViaAlGroups},   $\kappa $ will denote  an  arbitrary (associative)
unital commutative ring and unadorned $\otimes$ will stand for  $\otimes_{\kappa }$, as well as $\text{Hom}(V,W)$ will mean
$\text{Hom}_\kappa(V,W)$ for any $\kappa$-modules $V$ and $W$.

\begin{defi} \label{defi:twisted} Let $H$ be a Hopf $\kappa $-algebra, $A$ a unital $\kappa $-algebra
with unity element $\um .$ Let furthermore $\alpha:H\otimes A\to A$ and $\omega:H\otimes H\to A$ be two $\kappa $-linear
maps. We will write $\alpha(h\otimes a):=h\p a$, and $\omega (h\otimes l) : =\omega (h,l)$, where $a\in A$ and $h, l\in H$.

The pair $(\alpha,\omega)$ is called a \underline{\it twisted partial action}
of $H$ on $A$ if the fol\-lo\-wing conditions hold:
\begin{align}
1_H\p a&=a \label{unitpartial},\\
h\p (ab)&=\sum(h_{(1)}\p a)(h_{(2)}\p b) \label{productpartial},\\
\sum(h_{(1)}\p(l_{(1)}\p a))\omega(h_{(2)} ,l_{(2)})&=\sum\omega(h_{(1)} ,
l_{(1)})(h_{(2)}l_{(2)}\p a)\label{torcao},\\
\omega (h ,l) &= \sum \omega (h_{(1)} , l_{(1)})(h_{(2)}l_{(2)}\p\um )\label{cociclo},
\end{align}
for all $a,b\in A$ and $h,l\in H$.
\end{defi}

If $H$, $A$ and $(\alpha,\omega)$ satisfy Definition~\ref{defi:twisted},  then we shall also say that $(A, \cdot, \omega)$
is a \underline{twisted partial $H$-module algebra}.

\begin{prop} \label{FirstProp} If $(\alpha , \omega )$ is a twisted partial action, then the following identities hold:
\begin{eqnarray}
\omega (h,l)& =& \sum(h_{(1)}\cdot (l_{(1)} \cdot \um ))\omega (h_{(2)} ,l_{(2)})= \sum(h_{(1)}\p\um)\omega(h_{(2)} , l).\label{firstprop}\
\end{eqnarray}
\end{prop}

{\bf Proof:} The first identity is obtained  from (\ref{torcao}) by taking   $a=1:$
\[
\sum(h_{(1)}\cdot (l_{(1)} \cdot \um)\omega(h_{(2)} ,l_{(2)}) = \sum\omega(h_{(1)} , l_{(1)})(h_{(2)}l_{(2)}\p\um) 
\overset{\text{(\ref{cociclo})}}{=}\omega (h,l).
\]
For the second identity  notice that
\begin{eqnarray}
& \, & \sum(h_{(1)}\p\um)\omega(h_{(2)} , l) =\sum(h_{(1)}\p\um)(h_{(2)}\cdot (l_{(1)} \cdot \um ))\omega (h_{(3)} , l_{(2)}) =\nonumber\\
&\, & =\sum(h_{(1)}\cdot (l_{(1)} \cdot \um ))\omega (h_{(2)} ,l_{(2)})
= \omega (h,l), \nonumber
\end{eqnarray}  is obtained by using the first identity and  (\ref{productpartial}). \qed

\vt

We say that the map $\omega$ is {\it trivial}, if the following condition holds
\begin{align}
h\p(l\p \um)&=\omega (h ,l)=\sum(h_{(1)}\p\um)(h_{(2)}l\p\um)\
\label{cociclotrivial}\end{align}
for all $h,l\in H$. In this case, the twisted partial action $(\alpha,\omega)$ turns out to be a partial
action of $H$ on $A$, as introduced in \cite{CJ}. Indeed, if (\ref{cociclotrivial}) holds then the
condition (\ref{cociclo}) is superfluous, and for all $h,l\in H$ and
$a\in A$ we have:
\[
\begin{array}{ccl}
h\p(l\p a)&\overset{\text{(\ref{productpartial})}}{=}&\sum(h_{(1)}\p(l_{(1)}\p a))(h_{(2)}\p(l_{(2)}\p \um))
\overset{(\ref{cociclotrivial})}{=}\sum(h_{(1)}\p(l_{(1)}\p a))\omega(h_{(2)} ,l_{(2)})\\
&\overset{\text{(\ref{torcao})}}{=}&\sum\omega(h_{(1)} , l_{(1)})(h_{(2)}l_{(2)}\p a)
 \overset{(\ref{cociclotrivial})}{=} \sum(h_{(1)}\p \um)(h_{(2)}l_{(1)}\p \um)(h_{(3)}l_{(2)}\p a)\\
&\overset{\text{(\ref{productpartial})}}{=}&(h_{(1)}\p \um)(h_{(2)}l\p a).\
\end{array}
\]

Observe also that if $h\p \um =\varepsilon(h)\um$, for all $h\in H$, then the condition (\ref{cociclo}) is a
trivial consequence of the counit's properties and the $\kappa $-linearity of $\omega$, and so
we recover the classical notion of a twisted (global) action of $H$ on $A$ (see, for instance, \cite{Mont}).

\begin{ex}\label{ex:kG} This example is inspired by
\cite[Proposition 4.9]{CJ} and suits\footnote{If one assumes in Definition 2.1 of \cite{DES1} that each $D_g$
is generated by a central idempotent, then the definition below is more general, as neither the invertibility
in $D_g D_{gh}$ of each $w_{g,h}$ is required, nor the $2$-cocycle equality.} Definition 2.1 of \cite{DES1}.
An idempotent twisted partial action of a group $G$ on a $\kappa $-algebra $A$ is
a triple
$$\left(\{D_g\}_{g\in G}, \{\alpha_g\}_{g\in G},
\{w_{g,h}\}_{(g,h)\in G\times G}\right),$$ where for each $g\in G$,
$D_g$ is an ideal of $A$ generated by a central idempotent $1_g$ of
$A$, $\alpha_g:D_{g^{-1}}\to D_g$ is an isomorphism of unital
$\ka $-algebras, and for each $(g,h)\in G\times G$, $w_{g,h}$ is an element of $ D_g
D_{gh}$, and  the following statements are satisfied:
\begin{align}
1_e=\um \quad\text{and}\quad \alpha_e=I_A \label{pga1},\\
\alpha_g(\alpha_h(a1_{h^{-1}})1_{g^{-1}})w_{g,h}=w_{g,h}\alpha_{gh}(a1_{(gh)^{-1}}) \label{pga2},
\end{align}
for all $a\in A$ and $g,h\in G$, where $e$ denotes the identity element of $G$
and $I_A$ the identity map of $A$.

Let $\alpha:\ka G\otimes A\to A$ and $\omega:\ka G\otimes \ka G\to A$ be the
$\kappa $-linear maps given respectively by $\alpha(g\otimes
a)=\alpha_g(a1_{g^{-1}})$ and $\omega(g, h)=w_{g,h}$, for all
$a\in A$ and $g,h\in G$. This pair $(\alpha,\omega)$ is a twisted
partial action of $\ka G$ on $A$. Indeed, conditions (\ref{productpartial}) and (\ref{cociclo}) are
obvious since each $\alpha_g$ is multiplicative and each $w_{g,h}$
lives in $D_gD_{gh}$. Conditions (\ref{unitpartial}) and (\ref{torcao}) follow
easily from (\ref{pga1}) and (\ref{pga2}) respectively. Notice in addition that  $g \cdot \um=1_g$
is central, for all $g\in G$.

Conversely, consider a twisted partial action $(\alpha,\omega)$
of $\ka G$ on A and set $\alpha(g\otimes a)=g\p a$ and $\omega_{g,h}=
\omega(g,h)$, for all $g, h\in G$ and $a\in A$. By (\ref{productpartial}) we have that
$1_g:=g\p\um$ is an idempotent of $A$ and by (\ref{cociclo}) and Proposition~\ref{FirstProp}
$\omega_{g,h}\in (1_gA)\cap(A1_{gh})$, for all $g,h\in G$. Since
by (\ref{unitpartial}) $1_e=\um$, we have $\omega_{g,g^{-1}}\in 1_gA$ for all $g\in G$.

Now assume, in addition, that $1_g$ is central in $A$ and $\omega_{g,g^{-1}}$
is invertible in $1_gA$, for all $g\in G$. Thus, $1_gA$ is a unital $\kappa $-algebra,
$\omega_{g,h}\in (1_gA)(1_{gh}A)$, $$g\p 1_{g^{-1}}=g\p(g^{-1}\p\um)
\overset{(\ref{torcao})}{=}\omega_{g,g^{-1}}\um
\omega_{g,g^{-1}}^{-1}=1_g$$ and $$g\p(1_{g^{-1}}a)\overset{(\ref{productpartial})}{=}
(g\p1_{g^{-1}})(g\p a)=1_g(g.a)=(g\p\um)(g\p a)\overset{(\ref{productpartial})}{=}g\p a,$$
for all $g,h\in G$ and $a\in A$.

Hence, $\alpha$ induces by restriction a map of $\kappa $-algebras $\alpha_g:1_{g^{-1}}A\to 1_gA$,
given by $\alpha_g(1_{g^{-1}}a)=g\p(1_{g^{-1}}a)=g\p a$, for all $g\in G$ and $a\in A$.
Furthermore, it follows from (\ref{torcao}) that $\alpha_g\circ\alpha_{g^{-1}}(1_g a)=
\omega_{g,g^{-1}}(1_ga)\omega_{g,g^{-1}}^{-1}$, that is, $\alpha_g\circ\alpha_{g^{-1}}$
is an inner automorphism of $1_gA$. In particular, $\alpha_{g^{-1}}$
is injective and $\alpha_g$ is surjective, for all $g\in G$. Consequently,
$\alpha_g$ is an isomorphism and $1_gA=g\p A$, for all $g\in G$. Finally, (\ref{unitpartial}) and (\ref{torcao})
imply the above conditions (\ref{pga1}) and (\ref{pga2}) respectively. Therefore,
$$\left(\{ g\cdot A \}_{g\in G}, \{ g\cdot \underline{\,}\}_{g\in G},
\{\omega(g,h)\}_{(g,h)\in G\times G}\right)$$ is a twisted partial action of $G$ on $A,$ as defined above.\qed
\end{ex}


\begin{ex}{(Induced twisted partial action.)}\label{induced}

Let  $B$ be a unital $\kappa $-algebra measured by an action $\beta: H \otimes B \rightarrow B$, denoted by
$\beta (h,b) = h \rhd b$, which is twisted by a map $u: H \otimes H \rightarrow B$, i.e.,
\begin{align}
h \rhd (ab) & = \sum (h_{1} \rhd a)(h_{2} \rhd b),\\
h \rhd 1_B &= \varepsilon (h) 1_B,\\
\sum  (h_{(1)} \rhd (k_{(1)} \rhd a )) u(h_{(2)},k_{(2)}) &= \sum u(h_{(1)},k_{(1)}) (h_{(2)}k_{(2)} \rhd a), \label{twistedglobal}
\end{align} for all $h, k \in H$ and $a, b\in B.$ Assume furthermore that
\begin{equation}\label{1 cutuca a}
1_H \rhd a = a
\end{equation} for all $a\in A.$  Here  $u$ is neither supposed to be convolution invertible,  nor to satisfy the  $2$-cocycle equality.

Suppose that $\um$ is a non-trivial central idempotent of $B$, and let $A$ be the ideal generated by $\um$. Given $a \in A, h \in H$,
define a map $\cdot : H \otimes A \rightarrow A$ by
\begin{equation}\label{cotucapartial}
h \cdot a = \um (h \rhd a).
\end{equation}
 It is clear that (\ref{1 cutuca a}) implies  (\ref{unitpartial}),  and (\ref{productpartial}) follows from the fact that $\um b = b \um$
 for all $b \in B$. We still   have to define a map $\omega : H \ot H \rightarrow A$.

 From equation (\ref{twistedglobal}) we obtain
\begin{eqnarray*}
\sum (h_{(1)} \cdot (k_{(1)} \cdot a)) u(h_{(2)} , k_{(2)}) & = & \sum \um (h_{(1)} \rhd \um (k_{(1)} \rhd a)) u(h_{(2)} ,k_{(2)})\\
 & = & \sum (h_{(1)} \cdot \um)(h_{(2)} \rhd (k_{(1)} \rhd a)) u(h_{(3)} , k_{(2)}) \\
 & = & \sum (h_{(1)} \cdot \um)u(h_{(2)} , k_{(1)}) (h_{(3)}k_{(2)} \rhd a) \\
 & = & \sum (h_{(1)} \cdot \um)u(h_{(2)} , k_{(1)}) (h_{(3)}k_{(2)} \cdot a),
  \end{eqnarray*}
where the last equality follows from the fact that $ \sum (h_{(1)} \cdot \um)u(h_{(2)} , k)$ lies in $A$. In particular, for $a = \um$ we obtain
\begin{equation}\label{passagem}
\sum (h_{(1)} \cdot (k_{(1)} \cdot \um)) u(h_{(2)} , k_{(2)}) = \sum (h_{(1)} \cdot \um)u(h_{(2)} , k_{(1)}) (h_{(3)}k_{(2)} \cdot \um),
\end{equation}
which, in view of conditions (\ref{torcao}) and (\ref{cociclo}), suggests to define $\omega$ by
\begin{equation}
\omega (h , k)  =  \sum (h_{(1)} \cdot \um) u(h_{(2)} ,k_{(1)})  (h_{(3)}k_{(2)} \cdot \um) \label{def.omega}.
\end{equation}

\noindent With $\omega$ thus defined, (\ref{torcao}) and (\ref{cociclo}) are clearly satisfied,  and $(A, \cdot, \omega)$ is a twisted partial $H$-module algebra.


\vt
In particular, when $B$ is an $H$-module algebra, i.e., when $u$ is the trivial cocycle $u(h,k) = \varepsilon(h) \varepsilon(k) 1_B$,
it follows from (\ref{passagem}) that $\omega$ is also trivial, i.e. $\omega $ satisfies  (\ref{cociclotrivial}).
Therefore, in this case $A$ becomes a partial $H$-module algebra as defined in \cite{CJ}.  \qed
\end{ex}


Given any two $\kappa$-linear maps $\alpha:H\otimes A\to A$, $h\otimes
a\mapsto h\p a$, and $\omega:H\otimes H\to A$, we can define on the
$\kappa$-module $A\otimes H$ a product, given by the
multiplication
$$(a\otimes h)(b\otimes l)=\sum a(h_{(1)}\p b)\omega(h_{(2)} , l_{(1)})\otimes
h_{(3)}l_{(2)},$$ for all $a,b\in A$ and $h,l\in H$. Write  $A\#_{(\alpha,\omega)} H =(A\otimes H) (\um \otimes 1_H ).$
It is readily seen that this  corresponds   to the $\kappa $-submodule of
$A\otimes H$ generated by the elements of the form $a\# h:=\sum
a(h_{(1)}\p\um)\otimes h_{(2)}$, for all $a\in A$ and $h\in H$.

\vu

In general $A\otimes H$, with this above defined product, is neither associative nor unital. The following proposition
gives necessary and sufficient conditions
under which $A\otimes H$ (so, also $A\#_{(\alpha,\omega)} H$) is
associative and $\AH$ is unital with $\um\#
1_H=\um\otimes 1_H$ as the identity element. This proposition is a
generalization of \cite[Lemmas 4.4 and 4.5]{BCM} to the setting of
twisted partial Hopf algebra actions.

\vq

\begin{prop}
Let $A$ be a unital $\kappa $-algebra, $H$ a Hopf $\ka $-algebra, $\omega:H\otimes H\to A$ and
$\alpha:H\otimes A\to A$, $h\otimes a\mapsto h\p a$, two $\kappa $-linear
maps satisfying the conditions (\ref{unitpartial}) , (\ref{productpartial}) and (\ref{cociclo}).
\begin{enumerate}
\item[(i)]   $\um\# 1_H$ is the unity of $\AH$ if and
only if, for all $h\in H$,
\begin{align}
\omega(h, 1_H)=\omega(1_H , h)=h\p\um. \label{8}
\end{align}

\item[(ii)]  Suppose that $\omega(h, 1_H)=h\p\um$, for all $h\in
H$. Then $A\otimes H$ is associative if and only if the condition
(\ref{torcao})  holds and, for all $h,l,m\in H$,
\begin{align}
\sum(h_{(1)}\p\omega(l_{(1)} , m_{(1)}))\omega(h_{(2)} , l_{(2)}m_{(2)})&=
\sum\omega(h_{(1)}, l_{(1)})\omega(h_{(2)} l_{(2)} , m).\label{9}
\end{align}
\end{enumerate}
\end{prop}

\nd{\bf Proof.}\,\, The proof is quite similar to that of
\cite[Lemmas 4.4 and 4.5]{BCM}.

\vu

(i) Assume that $\omega(h,1_H)=\omega(1_H,h)=h\p\um$. Then
\[
\begin{array}{ccl}
(\um\# 1_H)(a\# h)&=& (\um \otimes 1_H) \left( \sum a (h_{(1)} \cdot \um ) \otimes h_{(2)} \right) \\
&=& \sum \um(1_H\p(a(h_{(1)}\p \um))\omega(1_H ,h_{(2)})\otimes 1_H h_{(3)}\\
&=&\sum a(h_{(1)}\p\um)(h_{(2)}\p\um)\otimes h_{(3)}\\
& \overset{(\ref{productpartial})}{=} & \sum
a(h_{(1)}\p\um)\otimes h_{(2)} = a\# h
\end{array}
\]
and
\[
\begin{array}{ccl}
(a\# h)(\um\# 1_H)&=& \left( \sum a(h_{(1)} \cdot \um ) \otimes h_{(2)} \right)(\um \otimes 1_H ) \\
&=& \sum a(h_{(1)}\p\um)(h_{(2)}\p\um)\omega(h_{(3)} ,1_H)\otimes h_{(4)}1_H\\
&=&\sum a(h_{(1)}\p\um)(h_{(2)}\p\um)(h_{(3)}\p 1_A)\otimes h_{(4)}\\
&\overset{(\ref{productpartial})}{=}&\sum a(h_{(1)}\p\um)\otimes h_{(2)}= a\# h\
\end{array}
\]
for every $a\in A$ and $h\in H$.

\vu

Conversely, if $\um\# 1_H$ is the unity of $A\#_{(\alpha,\omega)} H$
then applying $I_A\otimes\varepsilon$ to the equalities
\[
\sum(h_{(1)}\p\um)\otimes h_{(2)}=\um\# h=(\um\# 1_H)(\um\# h)=
\sum\omega(1_H , h_{(1)})\otimes h_{(2)}
\]
 and
\[
\begin{array}{ccl}
\sum(h_{(1)}\p\um)\otimes h_{(2)}&=&\um\# h=(\um\#
h)(\um\#1_H)\\
&=& \left( \sum (h_{(1)} \cdot \um )\otimes h_{(2)} \right) (\um \otimes 1_H ) \\
&=&\sum(h_{(1)}\p\um)(h_{(2)}\p\um)\omega(h_{(3)} , 1_H)\otimes
h_{(4)}\\
&\overset{(\ref{productpartial})}{=}&\sum(h_{(1)}\p\um)\omega( h_{(2)} ,1_H)\otimes h_{(3)}\\
&\overset{(\ref{cociclo})}{=}&\sum\omega(h_{(1)} ,1_H)\otimes h_{(2)},
\end{array}
\]
we obtain
\[
h\p\um=\sum (h_{(1)}\p\um)\varepsilon(h_{(2)})=\sum\omega(1_H,h_{(1)})\varepsilon(h_{(2)})=\omega(1_H,h)
\]
and
\[
h\p\um=\sum (h_{(1)}\p\um)\varepsilon(h_{(2)})=\sum\omega(h_{(1)} ,1_H)\varepsilon(h_{(2)})=\omega(h ,1_H).
\]

\vd

(ii) Assume that (\ref{torcao}) and (\ref{9}) hold. Then, for all $a,b,c\in A$ and
$h,l,m\in H$ we have:

\vu

$ (a\otimes h)[(b\otimes l)(c\otimes m)]$

\vspace{-.6cm}

\[
\begin{array}{ccl}
&=& \sum a[h_{(1)}\p(b(l_{(1)}\p
c)\omega(l_{(2)},m_{(1)}))]\omega(h_{(2)},l_{(3)}m_{(2)})\otimes
h_{(3)}l_{(4)}m_{(3)}\\
&\overset{(\ref{productpartial})}{=}&\sum a(h_{(1)}\p b)[(h_{(2)}\p(l_{(1)}\p
c))(h_{(3)}\p\omega(l_{(2)},m_{(1)}))\omega(h_{(4)},l_{(3)}m_{(2)})] \\
& \, & \otimes  h_{(5)}l_{(4)}m_{(3)}\\
&\overset{(\ref{9})}{=}&\sum a(h_{(1)}\p b)(h_{(2)}\p(l_{(1)}\p
c))\omega(h_{(3)},l_{(2)})\omega(h_{(4)}l_{(3)},m_{(1)})\otimes h_{(5)}l_{(4)}m_{(2)}\\
&\overset{(\ref{torcao})}{=}&\sum a(h_{(1)}\p b)\omega(h_{(2)},l_{(1)})(h_{(3)}l_{(2)}\p
c)\omega(h_{(4)}l_{(3)},m_{(1)})\otimes h_{(5)}l_{(4)}m_{(2)}\\
&=&[(a\otimes h)(b\otimes l)](c\otimes m).\
\end{array}
\]

Conversely, we have by assumption that
\[
(\um\otimes h)[(\um\otimes
l)(a\otimes 1_H)]=[(\um\otimes h)(\um\otimes l)](a\otimes 1_H)
\]
 and
\[
(\um\otimes h)[(\um\otimes l)(\um\otimes m)]=[(\um\otimes
h)(\um\otimes l)](\um\otimes m),
\]
for all $a\in A$ and $h,l,m\in H$.

Using mainly condition (\ref{cociclo}) (and the hypothesis on $\omega$ in the
first case only) one easily obtains, by a straightforward calculation,
from the first equality:
\[
(h_{(1)}\p(l_{(1)}\p a))\omega(h_{(2)} ,l_{(2)})\otimes
h_{(3)}l_{(3)}=\omega(h_{(1)} , l_{(1)})(h_{(2)}l_{(2)}\p a)\otimes h_{(3)}l_{(3)},
\]
 and from the second:
\begin{eqnarray} &\, & (h_{(1)}\p\omega(l_{(1)} , m_{(1)}))\omega(h_{(2)} , l_{(2)}m_{(2)})\otimes h_{(3)}l_{(3)}m_{(3)} \nonumber \\
& \, & = \omega(h_{(1)} , l_{(1)})\omega(h_{(2)}l_{(2)} , m_{(1)})\otimes
h_{(3)}l_{(3)}m_{(2)}.\nonumber
\end{eqnarray}

Now, applying $I_A\otimes\varepsilon$ in both sides of these
equalities the conditions (\ref{torcao}) and (\ref{9}) follow respectively. \qed

\vt

Given a twisted partial action $(\alpha,\omega)$ of a Hopf
$\ka $-algebra $H$ on a $\ka $-algebra $A$, the $\ka $-algebra
$\AH$ is called a {\it crossed product by a
twisted  partial action} (shortly, a {\it partial crossed product})
if the additional conditions (\ref{8}) and (\ref{9}) hold.

\vd

In order to establish some notation, we give the following lemma.

\begin{lemma} In $A\#_{(\alpha,\omega)} H$ we have the following identities:
\begin{enumerate}
\item[(i)] $a\# h=\sum a(h_{(1)} \cdot  \um )\# h_{(2)}$.
\item[(ii)] $(a\# h)(b\# k)=\sum a(h_{(1)} \cdot b )\omega (h_{(2)} ,k_{(1)}) \# h_{(3)} k_{(2)}$.
\end{enumerate}
\end{lemma}

\nd{\bf Proof.}\,\, Item (i) is straightforward,
\begin{eqnarray*}
\sum a(h_{(1)} \cdot \um )\# h_{(2)} &=& \sum a(h_{(1)} \cdot \um ) (h_{(2)}\cdot \um ) \otimes h_{(3)}\\
&=& \sum a(h_{(1)} \cdot \um )\otimes h_{(2)} =a\# h.
\end{eqnarray*}

For  item (ii), we have
\begin{eqnarray*}
(a\# h)(b\# k) & = & \left( \sum a(h_{(1)} \cdot  \um )\otimes  h_{(2)} \right) \left( \sum b(k_{(1)} \cdot \um )\otimes k_{(2)} \right) \\
& =& \sum a(h_{(1)} \cdot \um ) (h_{(2)} \cdot (b(k_{(1)}\cdot \um )))\omega ( h_{(3)},k_{(2)}) \otimes h_{(4)}k_{(3)} \\
& =& \sum a(h_{(1)} \cdot (b(k_{(1)}\cdot \um  )))\omega ( h_{(2)},k_{(2)}) \otimes h_{(3)}k_{(3)} \\
&=& \sum a(h_{(1)} \cdot b)(h_{(2)}\cdot (k_{(1)}\cdot \um  ))\omega ( h_{(3)},k_{(2)}) \otimes h_{(4)}k_{(3)} \\
&\overset{(\ref{firstprop})}{=}& \sum a(h_{(1)} \cdot b)\omega ( h_{(2)},k_{(1)}) \otimes h_{(3)}k_{(2)} \\
&\overset{(\ref{cociclo})}{=}& \sum a(h_{(1)} \cdot b)\omega ( h_{(2)},k_{(1)}) (h_{(3)}k_{(2)} \cdot \um ) \otimes h_{(4)}k_{(3)} \\
&=& \sum a(h_{(1)} \cdot b)\omega ( h_{(2)},k_{(1)}) \# h_{(3)}k_{(2)}.
\end{eqnarray*}\qed

If, in particular, $\omega$ is trivial then the multiplication in
$\AH$ becomes
\[
\begin{array}{ccl}
(a\# h)(b\# l)&=&(\sum a(h_{(1)}\p\um)\otimes h_{(2)})(\sum
b(l_{(1)}\p\um)\otimes l_{(2)})\\
&=&\sum a(h_{(1)}\p\um)(h_{(2)}\p(b(l_{(1)}\p\um)))\omega(h_{(3)} , l_{(2)})\otimes
h_{(4)}l_{(3)}\\
&\overset{(\ref{cociclotrivial})}{=}&\sum a(h_{(1)}\p\um)(h_{(2)}\p(b(l_{(1)}\p\um)))(h_{(3)}\p(l_{(2)}\p\um))\otimes
h_{(4)}l_{(3)}\\
&\overset{(\ref{productpartial})}{=}&\sum a(h_{(1)}\p(b(l_{(1)}\p\um))\otimes h_{(2)}l_{(2)}\
\end{array}
\]
for all $a,b\in A$ and $h,l\in H$. Consequently, in this case
we recover the partial smash product introduced in \cite{CJ}.

\vt

\begin{remark} \label{firstremark}\,\ If $(\alpha,\omega)$ is a twisted partial
action of $\ka G$ on a $\ka $-algebra $A$ arisen from a twisted partial
action of a group $G$ on $A$, as defined in Example~\ref{ex:kG}, and the
conditions (\ref{8}) and (\ref{9}) also hold in this case, then
$A\#_{(\alpha,\omega)}\ka G$ is an slight generalization of the partial
crossed product introduced in \cite{DES1}.
\end{remark}

\begin{ex}
Consider an induced partial twisted $H$-module structure as given in Example~\ref{induced}, and suppose that
the map $u: H \rightarrow A$ is a  \emph{normalized cocycle}, i.e., assume that
\begin{eqnarray}
\sum (h_{(1)} \rhd u(k_{(1)},l_{(1)}))u(h_{(2)},k_{(2)}l_{(2)}) & = & u(h_{(1)},k_{(1)}) u(h_{(2)}k_{(2)},l),
 \label{leidoscociclos}
\\ u(h , 1_H) & = &u(1_H , h) = \varepsilon(h) 1_B
\end{eqnarray}
for all $h,k,l \in H$. It is clear that the induced map
$\omega$ (see equality (\ref{def.omega}))
 satisfies condition (\ref{8}), and we will show that (\ref{9}) is also satisfied. In what follows, note that
 $\sum(h_{(1)} \cdot \um)(h_{(2)} \cdot x) = \sum(h_{(1)} \cdot x)(h_{(2)} \cdot \um)$,  and that if $a,b \in A$, then $a(h \cdot b)
 = a\um(h \rhd b) = a(h \rhd b)$. Using equality (\ref{leidoscociclos}),
\begin{eqnarray*}
& & \sum (h_{(1)} \cdot \omega(l_{(1)} , m_{(1)}) ) \omega(h_{(2)} , l_{(2)}m_{(2)}) = \\
& = & \sum(h_{(1)} \cdot [(l_{(1)} \cdot \um ) u(l_{(2)} , m_{(1)}) (l_{(3)}m_{(2)}\cdot \um)](h_{(2)} \cdot \um) \times \\
& & \times u(h_{(3)} , l_{(4)}m_{(3)}) (h_{(4)}l_{(5)}m_{(4)} \cdot \um) \\
& = & \sum(h_{(1)} \cdot(l_{(1)}\cdot \um))( (h_{(2)} \cdot u(l_{(2)} , m_{(1)})) (h_{(3)} \cdot (l_{(3)}m_{(2)} \cdot \um)) \times \\
& & \times (h_{(4)} \cdot \um) u(h_{(5)} , l_{(4)}m_{(3)})(h_{(6)}l_{(5)}m_{(4)} \cdot \um) \\
& = & \sum\underbrace{(h_{(1)} \cdot(l_{(1)} \cdot \um)(h_{(2)} \cdot \um)}( (h_{(3)} \cdot u(l_{(2)} , m_{(1)})) \times \\
& & \times (h_{(4)} \cdot (l_{(3)}m_{(2)} \cdot \um))
 u(h_{(5)} , l_{(4)}m_{(3)})(h_{(6)}l_{(5)}m_{(4)} \cdot \um) \\
& = & \sum(h_{(1)} \cdot(l_{(1)} \cdot \um)) (h_{(2)} \cdot u(l_{(2)} , m_{(1)})) \times \\
& & \times \underbrace{(h_{(3)} \cdot (l_{(3)}m_{(2)} \cdot \um)) u(h_{(4)} , l_{(4)}m_{(3)})} (h_{(5)}l_{(5)}m_{(4)} \cdot \um) 
\end{eqnarray*}
\begin{eqnarray*}
& \overset{\text{(\ref{passagem})}}{=} & \sum(h_{(1)} \cdot(l_{(1)} \cdot \um)) (h_{(2)} \cdot u(l_{(2)} , m_{(1)})) (h_{(3)} \cdot \um) u(h_{(4)} , l_{(3)}m_{(2)}) \times \\
& & \times (h_{(5)}l_{(4)}m_{(3)} \cdot \um) \\
& = & \sum(h_{(1)} \cdot(l_{(1)} \cdot \um)) (h_{(2)} \cdot u(l_{(2)} , m_{(1)}))  u(h_{(3)}, l_{(3)}m_{(2)}) 
(h_{(4)}l_{(4)}m_{(3)} \cdot \um) \\
& = & \sum(h_{(1)} \cdot(l_{(1)} \cdot \um))  \underbrace{(h_{(2)} \rhd u(l_{(3)} , m_{(1)}))  u(h_{(3)} , l_{(4)}m_{(2)})}
(h_{(4)}l_{(5)}m_{(4)} \cdot \um) \\
& \overset{\text{(\ref{leidoscociclos})} }{=}   & \sum\underbrace{(h_{(1)} \cdot(l_{(1)} \cdot \um)) u(h_{(2)} , l_{(2)})} u(h_{(3)}l_{(3)} , m_{(1)})(h_{(4)}l_{(4)}m_{(2)} \cdot \um) \\
& \overset{\text{(\ref{passagem})}}{=} & \sum(h_{(1)} \cdot \um)u(h_{(2)} , l_{(1)})(h_{(3)}l_{(2)} \cdot \um) u(h_{(4)}l_{(3)} , m_{(1)})
(h_{(5)}l_{(4)}m_{(2)} \cdot \um) \\
& = & \sum\omega(h_{(1)}, l_{(1)})\omega(h_{(2)}, l_{(2)}m).
\end{eqnarray*}

\end{ex}

If one forms the usual crossed product $B\#_{u} H$, then it is easy to see that $\um \# 1_H$ is an idempotent of this algebra, that
\[
(\um \# 1_H) (B\#_{u} H) = A \otimes H
\]
and that
\[
(\um \# 1_H)(B\#_{u} H) (\um \# 1_H) = \AH.
\] \qed



\section{Examples of (twisted) partial actions via algebraic groups}\label{ExViaAlGroups}

In this section we use the relation between algebraic groups and commutative Hopf algebras (see \cite{Abe}, \cite{Water})
to explain a way of producing  examples of partial Hopf (co)actions. A concrete example is  elaborated to
which we attach a twisting resulting in a twisted
partial Hopf action.\\

 We shall extract our example from  central notions of the theory of algebraic groups such as maximal tori, Cartan
 subgroups and the Weyl group. Let $\ka$ be an algebraically closed field and let ${\bf G}$ be a linear algebraic
 group over $\ka,$  by which we mean  a subgroup of ${\rm GL}_n(\ka )$ (for some positive integer $n$), which is
 closed in the Zariski topology of   ${\rm GL}_n(\ka ).$ Let $T$ be a maximal torus in $G.$ We recall that a torus
 is a connected diagonalizable linear algebraic group.   Let $C$ be the centralizer of $T$ in ${\bf G}$ and $N$ be
 the normalizer of $T$ in ${\bf G}.$ Then $C$ is the Cartan subgroup of ${\bf G}$ and $W=N/C$ is the corresponding
 Weyl group (which is finite).  Then evidently $N$ acts on itself by left multiplication and this permutes the left
 cosets of $N$ by $C.$  Then taking  the Hopf algebra $H$ which corresponds to $N,$ its comultiplication
 $\Delta: H \otimes H \to H$ is  the right coaction of $H$ on itself, which corresponds to the  left action of $N$ on itself.
 Let $Y\subsetneq N$ by a union of some left cosets of $N$ by $C.$ Then $N$ acts (globally) on $N$ and only partially on $Y.$
 Then one may take a two-sided ideal $A$ in $H$ determined by $Y$ (see the concrete example below)  so that one comes to
 a partial coaction $\rho : A \to A \otimes H$ which is obtained by the restriction $\rho = (1_{A} \otimes 1_{H}) \Delta .$
 Then taking  a Hopf algebra $H_1$ such that there exist a pairing $\langle , \rangle \: H_1\otimes H \to \ka ,$   one
 can dualize to obtain a partial action of $H_1$ on $A,$ as given in \cite[Prop. 8]{AB}.  In particular, $H_1$ can be
 the finite dual of $H.$



 For a concrete example we take one of the most classical cases, in which  ${\bf G}=  {\rm GL}_n(\ka )$ and
 $T\subseteq {\rm GL}_n(\ka )$ is the group of all diagonal matrices of $GL_n(k).$ Then $T \cong (\ka ^*)^n ,$
 where $\ka ^*$ is the multiplicative group of the field $\ka $.  It is directly verified that in this case
 $C=T$ and $N $   is formed by the monomial matrices, that is, the matrices whose rows and columns have only
 one nonvanishing entry. The Weyl group $W$ can be identified with the   group of $n\times n$ permutation
 matrices, which is isomorphic to the symmetric group $S_n$.\\

The group $N$ is an algebraic group and is isomorphic to the semidirect product of $T$ by the action of the Weyl group
\[
N\cong T\rtimes W= T\rtimes S_n .
\]
Here, the left action of $S_n$ on $T$ is given by conjugation, whose net effect is the permutation of the diagonal
matrix entries. By the fact that all these groups are algebraic groups, one can associate to the action
\[
\begin{array}{rccl} \alpha : & S_n \times T & \rightarrow & T, \\
\, & (g,x) & \mapsto & g\cdot x=gxg^{-1}, \end{array}
\] a left coaction of the corresponding Hopf algebras.  It is a basic fact that the Hopf algebra which corresponds
to a finite group is the dual of the group algebra, i.e. in our case it is  $(\ka S_n)^*.$  It is also basic that the
algebra corresponding to $\ka ^*$ is the Hopf algebra of the Laurent  polynomials $\ka [t,t^{-1} ].$  Since tensor products
of Hopf algebras correspond to direct product of algebraic groups, it follows that  the Hopf algebra corresponding to
$T$ is $\ka [t,t^{-1} ]^{\otimes n}.$ Consequently there is a  left coaction of  $(\ka S_n)^*$  on
$\ka [t,t^{-1} ]^{\otimes n},$ which corresponds to the above action of $S_n$ on $T,$ i.e.
$\ka [t,t^{-1} ]^{\otimes n}$ turns out to be a left $(\ka S_n )^*$-comodule coalgebra. Since $N\cong  T\rtimes S_n , $
it follows  by \cite[p. 143,  p. 208]{Abe} that the Hopf algebra associated to the group $N$ is the co-semidirect product
\[
\ka [t,t^{-1}]^{\otimes n} \cosemi (\ka S_n)^*.\\
\]

A typical element of $\ka [t,t^{-1}]^{\otimes n}$ is a tensor polynomial of the form
\[
\sum_{N\in \mathbb{Z}} \sum_{k_1 +\cdots +k_n =N} \lambda_N t^{k_1} \otimes \ldots \otimes t^{k_n} .
\]
In order to  simplify the notation, write
\begin{equation}\label{notation1}
t_i = 1 \otimes \ldots \otimes 1 \otimes t \otimes 1\otimes \ldots \otimes 1,
\end{equation} where $t$ belongs to the $i$-copy of $\ka [ t, t\m].$   Then we have  $t_1^{k_1}\ldots t_n^{k_n} = t^{k_1} \otimes \ldots \otimes t^{k_n}$.
Since $S_n$ operates  on $T$ by permuting the entries, it follows by a direct verification  that  the left $(\ka S_n)^*$-coaction on $\ka [t,t^{-1}]^{\otimes n}$ is given by
\[
\delta (t_1^{k_1}\ldots t_n^{k_n}) =\sum_{g\in S_n} p_g \otimes t_{g\m (1)}^{k_1} \ldots t_{g\m (n)}^{k_n} .
\]
With this coaction, $\ka [t,t^{-1}]^{\otimes n}$ is a left $(\ka S_n)^*$-comodule coalgebra, and the comultipication  of  the cosemidirect product
\[
H = \ka [t,t^{-1}]^{\otimes n} \cosemi (\ka S_n)^*
\]  is given explicitly by
\begin{eqnarray}
\Delta (t_1^{k_1}\ldots t_n^{k_n} \otimes p_g ) &=& \sum_{s,f\in S_n} t_1^{k_1}\ldots t_n^{k_n} \otimes p_s p_f \otimes
t_{s\m (1)}^{k_1}\ldots t_{s\m (n)}^{k_n} \otimes p_{f^{-1}g} \nonumber \\
&=& \sum_{s\in S_n} t_1^{k_1}\ldots t_n^{k_n} \otimes p_s \otimes
t_{s\m (1)}^{k_1}\ldots t_{s\m (n)}^{k_n} \otimes p_{s^{-1}g} .\nonumber
\end{eqnarray}
The cosemidirect product acts on the right on itself by the comultiplication. In order to construct a partial coaction one can simply project over a two-sided ideal.\\

Let $X$ be a subset of $S_n$ which is not a subgroup. Write $L = \ka [t,t^{-1}]^{\otimes n}.$ Then evidently
\begin{equation}\label{idempotent}
e_X = 1_L \otimes (\sum _{g \in X} p_g)
\end{equation} is a central idempotent  in $H=  L \cosemi (\ka S_n)^*,$ and the algebra $A = e_X H  $ is a two-sided  ideal.  Write
$e_X \cdot $ for the map $ H \to A$ given by multiplication by $e_X .$ Then the  restriction $\rho : A \to A\otimes H,$ $\rho = (e_X \cdot \otimes I ) \circ \Delta , $
of $\Delta _H : H \to H \otimes H$ is a right partial coaction of $H$ given by

\begin{equation}\label{rho}
{\rho } (t_1^{k_1}\ldots t_n^{k_n} \otimes p_g)=  \sum_{s\in X} t_1^{k_1}\ldots t_n^{k_n} \otimes p_s \otimes
t_{s\m (1)}^{k_1}\ldots t_{s\m (n)}^{k_n} \otimes p_{s^{-1}g},
\end{equation} where $g \in X.$ Since $X \subseteq S_n$ is not a subgroup, it is readily seen that $\rho$ is not a (global) coaction
(if $X$ was a subgroup then one would have $\rho : A \to A \otimes A$).\\

We shall obtain a partial action  from a partial coaction   using Proposition 8 from \cite{AB}, which we recall  for reader's convenience:

\begin{prop}\label{prop:pairing}   Let  $H_1   $ and $H_2  $  be  two Hopf algebras  with a  pairing between them:
\[
\begin{array}{rccl} \langle , \rangle : & H_1 \otimes H_2 & \rightarrow & \ka  \\
\, & h\otimes p & \mapsto & \langle h ,p \rangle .
\end{array}
\]
Then a partial right $H_2$-comodule algebra $A$, acquires a structure of partial left $H_1$-module algebra by the partial action
\[
h\cdot a = \sum a^{[0]} \langle h,a^{[1]} \rangle ,
\]
where ${\rho}(a) =\sum a^{[0]} \otimes a^{[1]}$ is the partial right coaction of $H_2$ on $A$.
\end{prop}




Assume now that $\ka $ is an isomorphic copy of the complex numbers $\Cc ,$  and let    $\mathbb{S}^1 $
be the unit circle group. The elements of $\mathbb{S}^1 $  can be viewed as the complex    roots of $1,$
however we assume that $\ka $ and  $\mathbb{S}^1 \subseteq \Cc$
are disjoint  and consider   the   group algebra $  \ka \mathbb{S}^1$   of $\mathbb{S}^1$ over $\ka ,$
so that the roots of unity $\chi \in \mathbb{S}^1 $ are linearly independent over $\ka .$ Then $S_n$ acts on
$ (\ka  \mathbb{S}^1 )^{\otimes n} $ by permutation of roots, which gives an action of the group  Hopf  algebra
$\ka  S_n $ on  $ (\ka \mathbb{S}^1 )^{\otimes n} .$ Note that  $\ka  S_n $ and  $ (\ka \mathbb{S}^1 )^{\otimes n} $
are   both cocommutative. Then we may consider   the smash product Hopf algebra
\[
H_1 = (\ka \mathbb{S}^1 )^{\otimes n} \rtimes \ka S_n.
\] Write
\begin{equation}\label{notation2}
\chi_{\theta_1 , \ldots \theta_n } = \chi_{\theta_1} \otimes \ldots \otimes \chi_{\theta_n } \in (\ka \mathbb{S}^1 )^{\otimes n},
\end{equation} where $\chi_{\theta_i} \in \mathbb{S}^1 $ is the root of $1$ whose angular coordinate is $\theta _i $ and which
belongs to the  $i$-factor of $(\ka \mathbb{S}^1 )^{\otimes n}.$   Then evidently the elements
$ \chi_{\theta_1 , \ldots \theta_n } \otimes u_g $ $(g\in S_n)$ form a $\ka $-basis of   $H_1.$
With this notation define the map $\langle , \rangle : H_1 \otimes H \to \ka ,$ by setting
\[
\langle \chi_{\theta_1 , \ldots \theta_n } \otimes u_g , t_1^{k_1} \ldots t_n^{k_n} \otimes p_s \rangle =\delta_{g,s}
\exp\{ik_1 \theta_1 \} \ldots \exp \{ ik_n \theta_n \},
\] where $ \delta_{g,s}  $ is the Kronecker delta and $i^2=-1.$  It is an easy straightforward verification that this defines a
pairing of Hopf algebras. Observe that it is non-degenerate, however we do not need to use this property.\\

Now using the coaction $\rho : A \to A\times H$ we obtain  by  Proposition~\ref{prop:pairing}   a partial action
$H_1 \times A \to A.$ To specify it, take $h =   \chi_{\theta_1 , \ldots \theta_n } \otimes u_g \in H_1$ and
$ a= t_1^{k_1} \ldots t_n^{k_n} \otimes p_s   \in A$ and check   by the formula in Proposition~\ref{prop:pairing} that the partial action is explicitly given by

\begin{align*}  & h \cdot a =  \sum _{f \in X}  t_1^{k_1} \ldots t_n^{k_n} \otimes p_f  \;
\langle \chi_{\theta_1 , \ldots \theta_n } \otimes u_g , t_{f\m(1)}^{k_ 1} \ldots t_{f\m(n)}^{k_n} \otimes p_{f\m s} \rangle = \\
& \sum_{f\in X}  t_1^{k_1} \ldots t_n^{k_n} \otimes p_f  \; (\delta_{g,f\m s}  \exp\{ik_1 \theta_{f\m(1)} \} \ldots \exp \{ ik_n \theta_{f\m(n)} \})
 =\\ &  \exp\{ik_1 \theta_{ g\m s(1)} \} \ldots \exp \{ ik_n \theta_{ g\m s (n)} \}  \; t_1^{k_1} \ldots t_n^{k_n} \otimes p_{ s\m g},
\end{align*} where $g\in S_n$ and $x\in X.$

 Since any finite group $G$ can be seen as a subgroup of $S_n$  for some $n,$ we may replace in the above considerations $S_n$ by an
 arbitrary finite group $G$ as follows.  Fix a monomorphism $G \to  S_n$ so that $G$ will be considered as a subgroup of $S_n.$  Then the formula
\[
\delta (t_1^{k_1}\ldots t_n^{k_n}) =\sum_{g\in G} p_g \otimes t_{g\m (1)}^{k_1} \ldots t_{g\m (n)}^{k_n}
\] gives a  structure of a left $(\ka G)^*$-comodule coalgebra on $L= \ka [t,t^{-1}]^{\otimes n},$ and one can take the  cosemidirect product
 $$H_2 = L \cosemi (\ka  G)^* $$ with comultiplication given by

\begin{eqnarray}
\Delta (t_1^{k_1}\ldots t_n^{k_n} \otimes p_g )
&=& \sum_{s\in G} t_1^{k_1}\ldots t_n^{k_n} \otimes p_s \otimes
t_{s\m (1)}^{k_1}\ldots t_{s\m (n)}^{k_n} \otimes p_{s^{-1}g}  \;\; (g\in G). \nonumber
\end{eqnarray}  Clearly, $H_2 = e H,$ where $e= 1_L \otimes (\sum _{g \in G} p_g).$\\

Let now $X$ be an arbitrary subset of $G$ which is not a subgroup.  The element  $e_X$ defined by the formula (\ref{idempotent}) is obviously a central idempotent in $H_2$ and  the algebra $A' = e_X H_2  $ is a two-sided  ideal.  The restriction
${\rho}' : A' \to A' \otimes H,$  ${\rho }' = e_X \Delta _{H_2} $ of $\Delta _{H_2} : H_2 \to H_2 \otimes H_2$ is a right partial coaction of $H_2$ given by  exactly the same formula (\ref{rho}) which was used for $\rho.$ The elements of $G \subseteq S_n$ act on $ (\ka \mathbb{S}^1 )^{\otimes n}, $ as above,  by permutation of roots,  and we have     the smash product
$$H'_1 = (\ka \mathbb{S}^1 )^{\otimes n} \rtimes \ka  G.$$ Then the formula above which defined the left partial action of $H_1$ on $A$ gives a left partial action  $H'_1 \times A' \to A' :$
\begin{equation}\label{ex:ParAc}
 (\chi_{\theta_1 , \ldots \theta_n } \otimes u_g) \cdot   (t_1^{k_1} \ldots t_n^{k_n} \otimes p_s )  =
\exp\{ i(k_1 \theta_{ g\m s(1)} + \ldots + k_n \theta_{ g\m s (n)}) \}  \; t_1^{k_1} \ldots t_n^{k_n} \otimes p_{ s\m g},
\end{equation} where $g \in G$ and $s \in X \subseteq G.$\\

In order to turn the partial action (\ref{ex:ParAc}) into a twisted one take a finite group $G$ whose Schur Multiplier over $\ka =\Cc$ is not trivial. Then there exists a $2$-cocycle $\gamma : G \times G \to {\Cc}^{\ast}$ which is not a coboundary. The $2$-cocycle equality means that
\begin{equation}\label{GroupCocycle}
\gamma  (x,y) \gamma (xy, z) = \gamma (x, yz)  \gamma (y,z) \;\;\;\; \;\;\;\;  \forall x,y,z \in G.
\end{equation}  Assume also that $\gamma$ is normalized, i.e.
\begin{equation}\label{normalized}
\gamma(g,1) = \gamma (1,g) = 1.
\end{equation} For arbitrary
$h= \chi_{\theta_1 , \ldots \theta_n } \otimes u_g$ and $l= \chi_{{\theta}'_1 , \ldots {\theta }'_n } \otimes u_s$ in $H'_1$  set

\begin{equation}\label{ex:omega}
\omega ( h, l ) = \gamma (g,s)\;  ( h \cdot l \cdot {\mathbf{1}_{A'} } ).
\end{equation}   The fact that (\ref{ex:ParAc})  and (\ref{ex:omega}) define a twisted partial action of $H'_1 = (\ka \mathbb{S}^1 )^{\otimes n} \rtimes \ka  G$ on $A',$ which satisfies  (\ref{8}) and (\ref{9}),  will follow from the next easy:

\begin{prop}\label{ex:omega2} Let  $G$ be a finite group and $L$  be a  cocommutative Hopf algebra over a field   $\ka ,$ such that $L$  is a left $\ka  G$-module algebra. Suppose that there is a left partial action  of the smash product $H= L \rtimes \ka G$ on a $\ka $-algebra $A:$
$$H \otimes A \ni h\otimes a \mapsto h\cdot a \in A.$$ If $\gamma : G \times G \to \ka ^{\ast}$ is a normalized $2$-cocycle, then the map
\begin{equation}\label{ex:omega3}
\omega(h , m)= \sum \gamma(g,s) \; (h \cdot m \cdot 1_A),
\end{equation} where $h = \sum l\otimes g, m=\sum l'\otimes s \in  L \rtimes \ka G,$ turns the partial action  $H \otimes A \to A$ into a twisted one such that (\ref{8}) and (\ref{9}) are satisfied.
\end{prop}

\begin{proof}    One needs to check   (\ref{torcao}),  (\ref{cociclo}),  (\ref{8})  and (\ref{9}). It is  obviously enough to verify these properties for the elements   $h,  m, k \in H$  of the form  $h =  l\otimes g,$  $m =  l'\otimes s,$ and  $ k =  l''\otimes f,$ with $l. l', l''\in L, g, s, f, \in G.$ Recall from  \cite[p. 142]{Abe}  that
\begin{equation*}
\Delta _H (l \otimes g) = \sum ( l_{(1)} \otimes g) \otimes ( l_{(2)} \otimes g)
\end{equation*} Then using  (\ref{normalized})  and (\ref{GroupCocycle})   it is readily seen that the properties   (\ref{torcao}),  (\ref{cociclo}),  (\ref{8})  and (\ref{9}) are resumed respectively to the following equalities:
\begin{align*}
& \sum ( h _ {(1)} \cdot m_{(1)} \cdot a  ) ( h _ {(2)} \cdot m_{(2)} \cdot \um ) =   \sum  ( h_{(1)}  \cdot m_{(1)} \cdot \um  )   ( h_{(2)}  m_{(2)} \cdot a  )\;\;\;   \;\;\;\;\;\;\; (\forall a\in A),\\
&    h  \cdot m  \cdot \um    = \sum ( h_{(1)}  \cdot m_{(1)} \cdot \um  ) ( h_{(2)}   m_{(2)} \cdot \um  ),\\
&  h  \cdot 1_H \cdot \um    = 1_H \cdot h   \cdot \um  =   h  \cdot \um ,\\
&\sum( h_{(1)}  \cdot m_{(1)} \cdot k_{(1)} \cdot \um  ) ( h_{(2)}  \cdot   m_{(2)}k_{(2)} \cdot \um  ) = \sum ( h_{(1)}  \cdot m_{(1)} \cdot \um  ) ( h_{(2)}   m_{(2)} \cdot  k _{(2)} \cdot \um  ).
\end{align*} The first three equalities are immediate consequences of the definition of a (non-twisted) partial action. As to the last one, write
\begin{align*} & \sum( h_{(1)}  \cdot m_{(1)} \cdot k_{(1)} \cdot \um  ) ( h_{(2)}  \cdot   m_{(2)}k_{(2)} \cdot \um  ) =\\
&\sum  h  \cdot [(m_{(1)} \cdot k_{(1)} \cdot \um  ) (     m_{(2)}k_{(2)} \cdot \um  ) ]=\\
&\sum  h  \cdot [(m_{(1)}\cdot \um )( m_{(2)} k_{(1)} \cdot \um  ) (     m_{(3)}k_{(2)} \cdot \um  ) ]=\\
&\sum  h  \cdot [(m_{(1)}\cdot \um )( m_{(2)} k_{(1)} \cdot \um  ) ]=
\end{align*}
\begin{align*}
& h  \cdot  m \cdot  k \cdot \um  =\sum  (h_{(1)}   \cdot \um) ( h_{(2)} m \cdot  k \cdot \um ) =\\
&\sum  (h_{(1)}   \cdot \um) ( h_{(2)} m_{(1)} \cdot   \um )  ( h_{(3)} m_{(2)} \cdot  k \cdot \um )  =\\
&\sum ( h_{(1)}  \cdot m_{(1)} \cdot \um  ) ( h_{(2)}   m_{(2)} \cdot  k _{(2)} \cdot \um  ),
\end{align*}  which completes the proof. \end{proof}

\vt

 Note that if in the proposition above we do not assume  (\ref{GroupCocycle}) and   (\ref{normalized}), i.e.   we take an arbitrary map
$\gamma : G \times G \to {\ka}^{\ast},$ then  we obtain a twisted partial action which in general does not satisfy (\ref{8}) and  (\ref{9}).\\

The above example can be made more specific  by taking a concrete group $G.$ The smallest finite group with non-trivial Schur Multiplier is the Klein-four group $G = \langle a \rangle \times \langle b \rangle ,$ $a^2 = b^2=1.$
In this case the Schur Multiplier $M(G)$ has order $2,$ and a $2$-cocycle $\gamma $, which is not a coboundary, can be easily obtained by considering the covering group $G^*$ of $G,$ which is the quaternion group of order $8.$
  In order to obtain $\gamma $   one takes a function $\phi  : G \to G^*,$ which is  a choice of  representatives of cosets of $G^*$ by $G,$ and defines
$\tilde{\gamma }(g,s) = \phi( g) \phi(s) (\phi (gs))\m ,$  $g,s \in G.$  Denote by $\varphi : {\mathcal Z}(G^*) \to \langle -1\rangle  $ the isomorphism between the center of $G^*$ (which has order $2$) and  $\langle -1\rangle  \subseteq \Cc .$  Then $\gamma = \varphi \circ \tilde{\gamma} $ is a $2$-cocycle which is not a coboundary.  One  readily checks that  this gives the cocycle  $\gamma : G \times G \to \ka ^*$ with  $\gamma (g,1) = \gamma (1,g) =1,$ for all $g\in G$ and
$\gamma (a,a) = \gamma (a,ab) = \gamma (b,a) =\gamma (b,b)  =\gamma (ab,b)  =\gamma (ab,ab)=-1,$
$\gamma (a,b) = \gamma (b,ab) =\gamma (ab,a)=1.$\\

We resume the example of this section in the next:

\begin{prop}\label{Ex:AlgGrHopf} Let $\ka $ be an isomorphic copy of the complex numbers $\Cc $ and let $\mathbb{S}^1  \subseteq \Cc$ be the circle group, i. e the group of all complex roots of $1.$  Let, furthermore, $G$ be an arbitrary finite group seen as a subgroup of $S_n$ for some $n.$ Taking the action of  $G \subseteq S_n$  on $ (\ka \mathbb{S}^1 )^{\otimes n} $   by permutation of roots,  consider  the smash product Hopf algebra $$H'_1 = (\ka \mathbb{S}^1 )^{\otimes n} \rtimes \ka  G.$$  Let $X\subseteq G$ be an arbitrary subset which is not a subgroup, and consider the subalgebra $\tilde {A} =  (\sum _{g\in X} p_g) (\ka  G)^* \subseteq  (\ka  G)^*,$ and write $A'=  \ka [t,t^{-1}]^{\otimes n} \otimes \tilde{A}.$ Then with the notation established in (\ref{notation1}) and (\ref{notation2}), the formula
\begin{equation*}
 (\chi_{\theta_1 , \ldots \theta_n } \otimes u_g) \cdot   (t_1^{k_1} \ldots t_n^{k_n} \otimes p_s )  =
\exp\{ i(k_1 \theta_{ g\m s(1)} + \ldots + k_n \theta_{ g\m s (n)}) \}  \; t_1^{k_1} \ldots t_n^{k_n} \otimes p_{ s\m g},
\end{equation*} where $g \in G$ and $s \in X \subseteq G$   gives a left partial action  $\alpha : H'_1 \times A' \to A' .$  Assume  now  that the Schur Multiplier of $G$ is non-trivial and take a normalized  (see (\ref{normalized}) ) $2$-cocycle $\gamma : G \times G \to {\ka }^{\ast}$ which is not a coboundary. For arbitrary
$h= \chi_{\theta_1 , \ldots \theta_n } \otimes u_g$ and $l= \chi_{{\theta}'_1 , \ldots {\theta }'_n } \otimes u_s$ in $H'_1$  set
\begin{equation*}
\omega ( h, l ) = \gamma (g,s)\;  ( h \cdot l \cdot {\mathbf{1}_{A'} } ).
\end{equation*}   Then the pair $(\alpha , \omega )$ forms a twisted partial action of
$H'_1 = (\ka \mathbb{S}^1 )^{\otimes n} \rtimes \ka  G$ on $A',$ which satisfies  (\ref{8}) and (\ref{9}).
\end{prop}


\section{Symmetric Twisted Partial Actions}

In \cite{DES1} a twisted partial action of a group $G$ over a unital $\ka $-algebra $A$ was defined as a triple
\[
\left(\{D_g\}_{g\in G}, \{\alpha_g\}_{g\in G},
\{w_{g,h}\}_{(g,h)\in G\times G}\right),
\]
where for each $g, h \in G$,
$D_g$ is an ideal of $A$ and $w_{g,h}$ is a multiplier of $D_g D_{gh}$ with some properties. If each $D_g$ is  generated by a central idempotent $1_g,$ then, as we have seen in Example~\ref{ex:kG}, this matches our concept  of a partial action of the group Hopf algebra $\ka G$ over $A$, and in this case, $1_g =g\cdot 1_A$. Then, from now on, unless explicitly stated, we are going to consider only  partial actions of a Hopf algebra $H$ over some unital algebra $A$ such that the map ${\bf e}\in \mbox{Hom} (H,A)$, given by ${\bf e}(h)=(h\cdot \um )$, is central with respect to the convolution product. These partial actions are, in some sense, more akin to partial group actions.

\vu

The second point of interest in twisted partial group actions is the case where the cocycles $\omega_{g,h}$ are invertible in $D_g D_{gh}$, for all $g,h\in G$. If the group action is global, then every element $\omega_{g,h}$ is an invertible element in $A$, this is automatically translated into the Hopf algebra setting by saying that the cocycle $\omega \in \mbox{Hom}(H\otimes H ,A)$ is convolution invertible. In the partial case, we have to search more suitable conditions to replace the convolution invertibility for the cocycle.

Let   $A = (A, \cdot, \omega)$ be a twisted partial $H$-module algebra. From the definition it follows that $f_1(h,k) = (h \cdot \um)\varepsilon(k) $  and $f_2(h,k) = (hk \cdot \um) $ are both (convolution) idempotents in $Hom(H \ot H,A)$. We also have that ${\bf e}$ is an idempotent in $\Hom(H,A)$ (and $f_1(h,k) = {\bf e}(h) \varepsilon(k)$).

Let us assume that both $f_1$ and $f_2$ are central in $\Hom(H \ot H, A)$. In this case  condition (\ref{cociclo}) of the definition of a twisted partial action reads as
\begin{eqnarray*}
\sum \omega(h_{(1)} , k_{(1)})(h_{(2)}k_{(2)} \cdot \um)& = &  \sum (h_{(1)}k_{(1)} \cdot \um) \omega(h_{(2)} , k_{(2)})= \omega(h , k) ,
\end{eqnarray*} and by Proposition~\ref{FirstProp} one also has:
\begin{eqnarray*}
\sum \omega(h_{(1)} , k)(h_{(2)} \cdot \um) & = & \sum (h_{(1)} \cdot \um)\omega(h_{(2)} , k) = \omega(h , k). \\
\end{eqnarray*} Notice that this  actually says that $\omega$ is an element of the  ideal  $\langle f_1 * f_2 \rangle \subset \Hom(H \ot H,A )$ generated by $ f_1 * f_2$. Clearly $f_1 * f_2   $ is the unity element of  $\langle f_1 * f_2  \rangle .$ Observe also that  the centrality of $f _1$ evidently implies that of ${\bf e} \in \Hom(H,A).$

\begin{defi}\label{symm} Let   $A = (A, \cdot, \omega)$ be a twisted partial $H$-module algebra.
We will say that the partial action
is \emph{symmetric} if
\begin{enumerate}[\rm (i)]
\item $f_1$ and $f_2$ are central in $\Hom(H \ot H, A);$
\item $\omega$ is a normalized  cocycle which is an  \emph{invertible} element  of the ideal $\langle f_1 * f_2 \rangle \subset  \Hom(H \ot H,A)$, i.e., $\omega$ satisfies conditions (\ref{8}) and (\ref{9}) and has a convolution inverse $\omega'$ in $\langle f_1 * f_2 \rangle ;$
\item $\sum (h \cdot (k \cdot \um)) = \sum (h_1 \cdot \um)(h_2 k \cdot \um)$, for every $h,k \in H .$
\end{enumerate}
\end{defi}

We remark once more that to say that $\omega': H \otimes H \rightarrow A$ lies in $\langle f_1 * f_2 \rangle$ is equivalent to require the equalities:
\begin{equation}
\sum \omega'(h_{(1)} , k_{(1)})(h_{(2)} \cdot \um) =  \omega'(h , k) = \sum
\omega'(h_{(1)} , k_{(1)})(h_{(2)}k_{(2)} \cdot \um), \label{abs.omegalinha}
\end{equation}
and that $\omega'$ is the inverse of $\omega $ in  $\langle f_1 * f_2 \rangle$ if and only if
\begin{equation}
(\omega * \omega' )(h , k)  = (\omega' * \omega) (h , k) = \sum (h_{(1)} \cdot \um)(h_{(2)} k \cdot \um) . \label{omega.omegalinha}
\end{equation}

It readily follows from  (\ref{abs.omegalinha}) and (\ref{omega.omegalinha}) that $\omega '$ is also normalized, i.e. ${\omega}'(1_H,h) = {\omega}'(h,1_H) = h\cdot \um $ for all $h \in H.$

\vt

Multiplying equality (\ref{torcao}) on the right by $\omega'$ and using (iii) of Definition~\ref{symm},  we obtain
\begin{equation}
h \cdot (k \cdot a) = \sum \omega(h_{(1)} , k_{(1)})(h_{(2)}k_{(2)} \cdot a)\omega'(h_{(3)} , k_{(3)}) \label{h.k.a}
\end{equation}
for all $h,k \in H$ and $a \in A$, which is an expression analogous to that of global twisted actions of Hopf algebras and also of partial twisted actions of groups. It is easy to prove that if one assumes that the two first items of Definition~\ref{symm} and equality (\ref{h.k.a}) hold, then item (iii) of Definition~\ref{symm} follows.

 Formula (\ref{h.k.a}) also provides another equality for $\omega'$ which is similar to (\ref{torcao}). Multiplying  (\ref{h.k.a})  by
$\omega'$ on the left and using the centrality of ${\bf e},$ we obtain
\begin{eqnarray*}
& & \sum  \omega'(h_{(1)} , k_{(1)}) (h_{(2)} \cdot (k_{(2)} \cdot a)) =\nonumber \\
& & = (\omega' * \omega)(h_{(1)} , k_{(1)}) (h_{(2)}k_{(2)} \cdot a) \omega'(h_{(3)} , k_{(3)}) \\
& & = (h_{(1)} \cdot \um)(h_{(2)}k_{(1)} \cdot \um ) (h_{(3)}k_{(2)} \cdot a) \omega'(h_{(4)} , k_{(3)}) \\
& & = \sum(h_{(1)}k_{(1)} \cdot a) (h_{(2)} \cdot \um)\omega'(h_{(3)} , k_{(2)})\\
& & \overset{\text{(\ref{abs.omegalinha})}}{=} \sum(h_{(1)}k_{(1)} \cdot a) \omega'(h_{(2)} , k_{(2)}).
\end{eqnarray*} Therefore, $\omega'$ satisfies
\begin{equation}
\sum \omega'(h_{(1)} , k_{(1)}) (h_{(2)} \cdot (k_{(2)} \cdot a))   = \sum(h_{(1)}k_{(1)} \cdot a) \omega'(h_{(2)} , k_{(2)})
\end{equation}
for all $h,k \in H$ and $a \in A$.

We shall need expressions for $h \cdot \omega(h , k)$ and $h \cdot \omega'(h, k)$, and for this we prove first an intermediate result, which is interesting on its own.

\begin{lemma}\label{lema.semigrupo}
Let $\semi$ be a semigroup and let $v,e,e'$ be elements of $\semi$. If there is an element $v' \in \semi$ such that
\begin{equation}\label{3.igualdades.semigrupo}
vv' = e, \ \ v'v = e' \text{ and } \ \ v'e = v',
\end{equation}
then $v' \in \semi $ satisfying (\ref{3.igualdades.semigrupo}) is unique.
\end{lemma}
\begin{proof}
In fact, assume that $v'$ is a solution of (\ref{3.igualdades.semigrupo}). Then $v'$ also satisfies
\begin{equation}
e'v' = v' \label{4a.igualdade.semigrupo},
\end{equation}
because
\[
e'v' = (v'v)v'=v'(vv')=v'e=v'.
\]
Suppose that $v''$ is another solution. It follows from (\ref{3.igualdades.semigrupo}) and
(\ref{4a.igualdade.semigrupo}) that
\[
v'' = v''e = v''(vv') = (v''v)v' = e'v' = v'.
\]
\end{proof}

\begin{prop}
Let $(A,\cdot, (\omega,\omega'))$ be a symmetric twisted partial $H$-module algebra. Then
\begin{eqnarray}
h \cdot \omega(k , m) & = & \sum \omega(h_{(1)} , k_{(1)}) \omega(h_{(2)}k_{(2)} , m_{(1)}) \omega'(h_{(3)} , k_{(3)}m_{(2)}), \label{h.em.omega}\\
h \cdot \omega'(k , m) & = & \sum \omega(h_{(1)} , k_{(1)}m_{(1)}) \omega'(h_{(2)}k_{(2)} , m_{(2)}) \omega'(h_{(3)} , k_{(3)}). \label{h.em.omega.linha}
\end{eqnarray}
\end{prop}
\begin{proof}
To prove (\ref{h.em.omega}),  multiply (\ref{9}) by $\omega'$ on the right, obtaining
\begin{eqnarray}
& & \sum (h_{(1)} \cdot \omega (k_{(1)} , m_{(1)})) (\omega * \omega' )(h_{(2)} , k_{(2)}m_{(2)}) = \nonumber \\
& & = \sum
\omega (h_{(1)} , k_{(1)}) \omega (h_{(2)}k_{(2)} , m_{(1)})  \omega' (h_{(2)} , k_{(2)}m_{(2)})
\nonumber
\end{eqnarray}
Since the left hand side  equals
\begin{eqnarray*}
   & & \sum (h_{(1)} \cdot ( \omega(k_{(1)} , m_{(1)}))(h_{(2)} \cdot (k_{(2)}m_{(2)} \cdot \um)) = \\
   & = & h \cdot (\sum \omega(k_{(1)} , m_{(1)}) (k_{(2)}m_{(2)} \cdot \um)) =  h \cdot \omega(k,m) , \\
\end{eqnarray*} equation (\ref{h.em.omega}) follows.
The proof of (\ref{h.em.omega.linha}) is a bit more involved and uses Lemma~\ref{lema.semigrupo}. Consider $\Hom(H^{\otimes^3},A)$ as a multiplicative semigroup, and take the elements
\begin{eqnarray*}
v(h , k , m) & = & h \cdot \omega (k , m), \\
e(h , k , m) & = & e'(h \ot k \ot m) = (h \cdot (k \cdot (m \cdot \um))) =\nonumber\\ & = & \sum h \cdot [(k_{(1)} \cdot \um)(k_{(2)}m \cdot \um)].
\end{eqnarray*}
We will show that
\begin{eqnarray*}
v'(h , k , m) & = & h \cdot \omega' (k , m), \\
v''(h , k , m) & = & \sum \omega(h_{(1)} , k_{(1)}m_{(1)}) \omega'(h_{(2)}k_{(2)} , m_{(2)}) \omega'(h_{(3)} , k_{(3)})
\end{eqnarray*} satisfy
\beqnst
v * v' = e, \ \ v' * v = e, \ \ v' * e = v'
\eqnst
and
\beqnst
v * v'' = e, \ \ v'' * v = e, \ \ v'' * e = v'',
\eqnst
thus proving, via Lemma~\ref{lema.semigrupo}, that $v'=v''$. 

Keep in mind that $e$ can be written also as 	
\beqnast
e(h , k , m) & = & \sum h \cdot [(k_{(1)} \cdot \um)(k_{(2)}m \cdot \um)] \\
& = & \sum (h_{(1)} \cdot \um)(h_{(2)}k_{(1)} \cdot \um)(h_{(3)}k_{(2)}m \cdot \um).
\eqnast
The equalities involving $v'$ are straightforward. For instance, 
\[
(v'*e)(h,  k, m)  = 
 h [\cdot (\omega' (k , m))(k \cdot (m \cdot \um))] 
  =  h \cdot (\omega' (k , m)) =  v'(h \ot k \ot m).
 \]

As for $v''$, we first compute $v * v'',$ using (\ref{h.em.omega}), the centrality of ${\bf e}$ and (3) of Definition~\ref{defi:twisted}:
\begin{eqnarray*}
&& (v * v'')(h , k , m) =  \\
& = & \sum \underbrace{(h_{(1)} \cdot \omega(k_{(1)} , m_{(1)})) \omega(h_{(2)} , k_{(2)}m_{(2)})} \omega'(h_{(3)}k_{(3)} , m_{(3)}) \omega'(h_{(4)} , k_{(4)}) \\
& \overset{\text{(\ref{9})}}{=} & 
\sum  \omega(h_{(1)} , k_{(1)})\omega(h_{(2)}k_{(2)} , m_{(1)})
\omega'(h_{(3)} , k_{(3)}m_{(2)}) \omega'(h_{(4)} , k_{(4)})\\
& \overset{\text{\ref{symm}.iii,(\ref{omega.omegalinha})}}{=} & 
\sum  \omega(h_{(1)} , k_{(1)})(h_{(2)}k_{(2)} \p ( m \p \um))
 \omega'(h_{(3)} , k_{(3)})\\
& = & (h \cdot (k \cdot (m \cdot \um))) = e(h , k , m).
\end{eqnarray*}
 
Observe next that given $m \in H$, the linear function $\nu_{m}: H \ot H \vai A$ given by $h \ot k \mapsto \omega (h, k m)  $ lies in $\Hom (H \otimes  H ,A)$  and therefore commutes with $f_2$, i.e.

\begin{equation}\label{commutef2}
\sum \omega(h_{(1)}, k_{(1)}  m) (h_{(2)} k_{(2)} \cdot \um ) =  \sum (h_{(1)} k_{(1)} \cdot \um )  \omega(h_{(2)}, k_{(2)}  m) ,
\end{equation}   for all $h,k,m\in H.$ Then  we  calculate $v'' \ast  v$:

\begin{eqnarray*}
&& (v'' * v)(h \ot k \ot m) =\\
& = & \sum \omega(h_{(1)} , k_{(1)}m_{(1)}) \omega'(h_{(2)}k_{(2)} , m_{(2)}) \omega'(h_{(3)} , k_{(3)}) \, (h_{(4)} \cdot \omega(k_{(4)} , m_{(3)}))  \\
& \overset{\text{(\ref{h.em.omega})}}{=} & \sum \omega(h_{(1)} , k_{(1)}m_{(1)}) \omega'(h_{(2)}k_{(2)} , m_{(2)}) \underbrace{\omega'(h_{(3)} , k_{(3)}) \omega (h_{(4)} , k_{(4)})} \times \\
&& \times \omega (h_{(5)} k_{(5)} , m_{(3)}) \omega'(h_{(6)} , k_{(6)}m_{(4)})\\
& = & \sum \omega(h_{(1)} , k_{(1)}m_{(1)}) \omega'(h_{(2)}k_{(2)} , m_{(2)}) (h_{(3)} \cdot \um)(h_{(4)}k_{(3)} \cdot \um)
\times \\ && \times \omega (h_{(5)} k_{(4)} , m_{(3)}) \omega'(h_{(6)} , k_{(5)}m_{(4)})\\
& = & \sum \underbrace{\omega(h_{(1)} , k_{(1)}m_{(1)}) (h_{(2)} \cdot \um)}
\overbrace{\omega'(h_{(3)}k_{(2)} , m_{(2)}) (h_{(4)}k_{(3)} \cdot \um)}
\times \\
&& \times \omega (h_{(5)} k_{(4)} , m_{(3)})\omega'(h_{(6)} , k_{(5)}m_{(4)})\\
& = & \sum \omega(h_{(1)} , k_{(1)}m_{(1)}) \underbrace{\omega'(h_{(2)}k_{(2)} , m_{(2)}) \omega (h_{(3)} k_{(3)} , m_{(3)})}  \omega'(h_{(4)} , k_{(4)}m_{(4)})\\
& = & \sum \omega(h_{(1)} , k_{(1)}m_{(1)}) (h_{(2)}k_{(2)} \cdot \um)  \underbrace{(h_{(3)} k_{(3)} m_{(2)} \cdot \um)
\omega' (h_{(4)} , k_{(4)}m_{(3)})}\\
& = & \sum \omega(h_{(1)} , k_{(1)}m_{(1)}) (h_{(2)}k_{(2)} \cdot \um)\omega'(h_{(3)} , k_{(3)}m_{(2)})\\
& \overset{\text{(\ref{commutef2})}}{=} & \sum (h_{(1)}k_{(1)} \cdot \um)\omega(h_{(2)} , k_{(2)}m_{(1)}) \omega'(h_{(3)} , k_{(3)}m_{(2)})\\
& = & \sum (h_{(1)} \cdot \um)(h_{(2)}k_{(1)} \cdot \um)(h_{(3)}k_{(2)}m \cdot \um) =  e(h , k , m).
\end{eqnarray*}

And finally, the expression for $v'' * e$ can be obtained as follows
\begin{eqnarray*}
&&v'' * e (h , k , m) = \\
& = & \sum \omega(h_{(1)} , k_{(1)}m_{(1)}) \omega'(h_{(2)}k_{(2)} , m_{(2)}) \underbrace{\omega'(h_{(3)} , k_{(3)})  ( h_{(4)} \cdot \um) }
 \times \\
&& \times  \underbrace{ (h_{(5)}k_{(4)} \cdot \um)  (h_{(6)}k_{(5)}m_{(3)} \cdot \um) }  \\
& =& \sum \omega(h_{(1)} , k_{(1)}m_{(1)}) \omega'(h_{(2)}k_{(2)} , m_{(2)}) \omega'(h_{(3)} , k_{(3)})  (h_{(4)}k_{(4)} \cdot (m_{(3)} \cdot \um)) 
\\
&\overset{\text{(\ref{h.k.a})}}{=}& \sum \omega(h_{(1)} , k_{(1)}m_{(1)}) \omega'(h_{(2)}k_{(2)} , m_{(2)}) \underbrace{ \omega'(h_{(3)} , k_{(3)})  \omega(h_{(4)} , k_{(4)})} \times \\
&& \times (h_{(5)} k_{(5)} m_{(3)} \cdot \um)\omega'(h_{(6)} , k_{(6)})\\
& =& \sum \omega(h_{(1)} , k_{(1)}m_{(1)}) \omega'(h_{(2)}k_{(2)} , m_{(2)}) (h_{(3)} \cdot \um)
(h_{(4)}k_{(3)} \cdot \um)\times \\
 && \times(h_{(5)} k_{(4)} m_{(3)} \cdot \um)\omega'(h_{(6)} , k_{(5)})
\end{eqnarray*}
\begin{eqnarray*}
&=&\sum \underbrace{ \omega(h_{(1)} , k_{(1)}m_{(1)}) (h_{(2)} \cdot \um)}  \overbrace {\omega'(h_{(3)}k_{(2)} , m_{(2)})
(h_{(4)}k_{(3)} \cdot \um)  (h_{(5)} k_{(4)} m_{(3)} \cdot \um) } \times \\
 && \times   \omega'(h_{(6)} , k_{(5)})\\
& = & \sum \omega(h_{(1)} , k_{(1)}m_{(1)}) \omega'(h_{(2)}k_{(2)} , m_{(2)}) \omega'(h_{(3)} , k_{(3)}) = v''(h , k , m).
\end{eqnarray*}
Therefore, Lemma~\ref{lema.semigrupo} implies   (\ref{h.em.omega.linha}).   \end{proof}

\vu

\begin{ex}
Consider a twisted $H$-module algebra $B$ as in Example~\ref{induced}, and assume that the map $u: H \otimes H \rightarrow B$, which twists the action, is a normalized \emph{invertible} cocycle with convolution inverse $\uinv$. Suppose furthermore, that $B$ has a nontrivial central idempotent $\um$, and consider the twisted partial $H$-module structure on the ideal $A =\um B$ as it was done in Example~\ref{induced}: the partial action and the cocycle $\omega$ are defined by
\begin{eqnarray*}
h \cdot a & = & \um (h \rhd a) \\
\omega(h , k) & = & \sum (h_{(1)} \cdot \um) u(h_{(2)} , k_{(1)}) (h_{(3)}k_{(2)} \cdot \um).
\end{eqnarray*}
Suppose also that $f_1(h \ot k)=(h \cdot \um)\varepsilon(k)$ and $f_2(h \ot k) = (hk \cdot \um)$ are  central  in $\Hom(H \ot H,A)$. Under this hypothesis, the functions $h \ot k \mapsto \um u(h,k) $ and $h \ot k \mapsto \um u\m (h,k) $ commute with ${\bf e}$ and $f_2,$ and  it is obvious that
\beqnst
\omega'(h , k)  =  \sum (h_{(1)}k_{(1)} \cdot \um)u^{-1}(h_{(2)} , k_{(2)}) (h_{(3)} \cdot \um)
\eqnst
is the inverse of $\omega$ in $\langle f_1 * f_2 \rangle$. Note also that
\begin{eqnarray*}
h \cdot (k \cdot a) & = & \um  (h \rhd (\um (k \rhd a))) = \um  (h \rhd (\um (k \rhd a) \um))\\
 & = & \um \sum (h_{(1)} \rhd \um) (h_{(2)} \rhd (k \rhd a)) (h_{(3)} \rhd \um)\\
 & = & \um \sum (h_{(1)} \rhd \um) u(h_{(2)} , k_{(1)})  (h_{(3)}k_{(2)} \rhd a)  \uinv (h_{(4)},k_{(3)})(h_{(5)} \rhd \um)
\\
 & = & \um \sum (h_{(1)} \rhd \um) u(h_{(2)} , k_{(1)})  (h_{(3)}k_{(2)} \rhd \um)(h_{(4)}k_{(3)} \rhd a) \times \\
 && \times (h_{(5)}k_{(3)} \rhd \um) \uinv (h_{(6)} , k_{(4)})(h_{(7)} \rhd \um)\\
 & = & \sum  \underbrace{(h_{(1)} \cdot \um) u(h_{(2)} , k_{(1)})  (h_{(3)}k_{(2)} \cdot \um)}(h_{(4)}k_{(3)} \cdot a) \times \\
 && \underbrace{(h_{(5)}k_{(3)} \cdot \um) \uinv (h_{(6)} , k_{(4)})(h_{(7)} \cdot \um)}\\
& =& \sum \omega(h_{(1)} , k_{(1)})(h_{(2)}  k_{(2)} \cdot a)\omega'(h_{(3)} , k_{(3)}),
\end{eqnarray*}
and it follows that
\begin{eqnarray*}
h \cdot (k \cdot \um)
& = & \sum \omega(h_{(1)} , k_{(1)})(h_{(2)}  k_{(2)} \cdot \um)\omega'(h_{(3)} , k_{(3)})  \\
& = & \sum \omega(h_{(1)} , k_{(1)})\omega'(h_{(2)} , k_{(2)})  = \sum (h_{(1)} \cdot \um)(h_{(2)}k \cdot \um),\\
\end{eqnarray*} proving that we have a symmetric twisted partial $H$-module algebra.\qed
\end{ex}

Another important point arising in the context of symmetric twisted partial Hopf actions is to give criteria in order to decide whether two twisted partial actions give rise to the same crossed product. In the classical case, two crossed products are isomorphic if, and only if the associated twisted (global) actions can be transformed one into another by some kind of coboundary (see, for instance  \cite{Mont} for the main results  of the classical case). In the case of abelian groups, there is, indeed, a cohomology theory involved, and the cocycles performing the twisted actions are related by coboundaries. What we shall see now is an analogue of  Theorem~7.3.4 of \cite{Mont} for twisted partial Hopf actions, this result opens a window for a cohomological point of view of the twisted cocycles presented above.

\begin{thm} \label{the41}Let $A$ be a unital algebra and $H$ a Hopf algebra with two symmetric twisted partial actions on $A$, $h\otimes a \mapsto h\cdot a$, and $h\otimes a \mapsto h\bullet a$, with  cocycles $\omega$ and $\sigma$, respectively. Suppose that  there is an algebra isomorphism
\[
\Phi :\, A\#_{\omega} H \, \rightarrow \, A\#_{\sigma} H
\]
which is also a left $A$-module and right $H$-comodule map. Then there exists linear maps $u,v\in \Hom(H,A)$ such that, for all $h,k\in H$, $a\in A$,
\begin{enumerate}
\item[(i)] $u*v(h)=h\cdot \um $,
\item[(ii)] $u(h)=\sum u(h_{(1)})(h_{(2)}\cdot \um )=\sum (h_{(1)} \cdot \um )u(h_{(2)})$,
\item[(iii)] $h\bullet a=\sum v(h_{(1)}) (h_{(2)}\cdot a)u(h_{(3)})$,
\item[(iv)] $\sigma (h,k)=\sum v(h_{(1)})(h_{(2)}\cdot v(k_{(1)}))\omega (h_{(3)},k_{(2)}) u(h_{(4)}k_{(3)})$,
\item[(v)] $\Phi (a\#_{\omega} h) =\sum au(h_{(1)})\#_{\sigma} h_{(2)}$.
\end{enumerate}
Conversely, given maps $u,v\in \Hom(H,A)$ satisfying (i),(ii),(iii) and (iv), and in addition $u(1_H )=v(1_H )=\um$, then the map $\Phi$, as presented in (v), is an isomorphism of algebras.
\end{thm}

\begin{proof} ($\Rightarrow$) The left $A$-module structure on the crossed products is given by the left multiplication:
\[
a\blacktriangleright (b\# h)= (a\# 1_H )(b\# h)=ab\# h ,
\]
and the right $H$-comodule structure is given by $\rho =\mbox{I}_A \otimes \Delta$. Let $\Phi :\, A\#_{\omega} H \, \rightarrow \, A\#_{\sigma} H$ be the algebra isomorphism which also is a left $A$-module and right $H$-comodule map. Define $u,v\in \mbox{Hom}(H,A)$ as
\[
u(h)= (\mbox{I}_A \otimes \varepsilon )\Phi (1_A \#_{\omega} h) , \qquad \mbox{ and } \qquad  v(h)= (\mbox{I}_A \otimes \varepsilon )\Phi^{-1} (1_A \#_{\sigma} h) .
\]
Let us verify that the maps $u,v$, as defined above, satisfy the items (i) to (v). For the item (v) we have, for all $a\in A$ and $h\in H$
\begin{eqnarray*}
\Phi (a\#_{\omega} h) & =& a\blacktriangleright ((\Phi (\um \#_{\omega} h)
)   ) \\
&=& a\blacktriangleright \{  (\mbox{I}_A \otimes \varepsilon \otimes \mbox{I}_H) (\mbox{I}_A\otimes \Delta )\Phi (\um \#_{\omega} h) \}\\
&=& a\blacktriangleright \{   (\mbox{I}_A \otimes \varepsilon \otimes \mbox{I}_H) \Phi \otimes \mbox{I}_H )(\sum \um \#_{\omega} h_{(1)}) \otimes h_{(2)} )    \} \\
&=&a\blacktriangleright \{   (\mbox{I}_A \otimes \varepsilon ) \Phi(\sum \um \#_{\omega} h_{(1)}) \otimes h_{(2)} \}\\
&=& a\blacktriangleright ( \sum u(h_{(1)}) \#_{\sigma} h_{(2)})  =\sum au(h_{(1)}) \#_{\sigma} h_{(2)} .
\end{eqnarray*}
With a totally similar reasoning, we can conclude that
\[
\Phi^{-1} (a\#_{\sigma} h) =\sum av(h_{(1)}) \#_{\omega} h_{(2)} .
\] Notice that we readily obtain from the above that $u(1_H )= v(1_H) = \um.$

\vu

For  item (i) consider the expression
\begin{eqnarray*}
\sum (h_{(1)}\cdot \um )\#_{\omega} h_{(2)} &=& \um \#_{\omega} h =\Phi^{-1} (\Phi (\um
\#_{\omega} h))\\
&=& \Phi^{-1} (\sum u(h_{(1)} )\#_{\sigma} h_{(2)} )\\
&=& \sum u(h_{(1)}) v(h_{(2)}) \#_{\omega} h_{(3)} .
\end{eqnarray*}
Applying $(\mbox{Id}\otimes \varepsilon )$ on both sides, we obtain
\[
\sum u(h_{(1)} ) v(h_{(2)} ) =h\cdot \um .
\]
Analogously, we can conclude that
\[
\sum v(h_{(1)} )u( h_{(2)} ) =h\bullet \um .
\]

 Item (ii) is easily obtained by applying $\mbox{I}_A\otimes \varepsilon$ on both sides of the equality
\[
\sum u(h_{(1)})\#_{\sigma} h_{(2)}= \Phi (\um\#_{\omega} h)=\Phi (\sum (h_{(1)}\cdot \um )\#_{\omega} h_{(2)})=
\sum (h_{(1)}\cdot \um )u(h_{(2)}) \#_{\sigma} h_{(3)} .
\] The absorption of $h \cdot \um$ on the other side in (ii) comes from the fact that the twisted partial action  is symmetric.

In order to prove items (iii) and (iv), we use the fact that $\Phi^{-1}$ is an algebra morphism, as so is $\Phi $ either. Therefore
\[
\Phi^{-1} ((a\#_{\sigma }h )(b\# _{\sigma} k)) =\Phi^{-1}(a\#_{\sigma} h) \Phi^{-1}
(b\#_{\sigma} k) ,
\]
which gives
\begin{eqnarray*}
& &\sum a(h_{(1)}\bullet b )\sigma (h_{(2)}, k_{(1)}) v(h_{(3)}k_{(2)})\#_{\omega} h_{(4)} k_{(3)} =\\
&=& \sum av(h_{(1)})(h_{(2)}\cdot (bv(k_{(1)}))) \omega (h_{(3)}, k_{(2)})
\#_{\omega} h_{(4)}k_{(3)} .
\end{eqnarray*}
Applying $\mbox{I}_A\otimes \varepsilon$ on both sides, we get
\begin{equation}{\label{cobordos}}
\sum a(h_{(1)}\bullet b )\sigma (h_{(2)}, k_{(1)}) v(h_{(3)}k_{(2)}) =\sum av(h_{(1)})(h_{(2)}\cdot (bv(k_{(1)}))) \omega (h_{(3)}, k_{(2)}) .
\end{equation}
Using this formula for $a=\um $ and $k=1_H$ we obtain
\[
\sum (h_{(1)}\bullet b ) v(h_{(2)}) =\sum v(h_{(1)})(h_{(2)}\cdot b) .
\]
The expression (iii) is finally obtained multiplying convolutively on the right by $u$:
\[
\sum (h_{(1)}\bullet b ) v(h_{(2)}) u(h_{(3)})=\sum v(h_{(1)})(h_{(2)}\cdot b) u(h_{(3)}),
\]
and using the fact that $v*u (h)=h\bullet \um$. This gives
\[
h\bullet b=\sum v(h_{(1)}) (h_{(2)}\cdot b)u(h_{(3)}) .
\]

On the other hand, putting $a=b=\um$ in (\ref{cobordos}) we get
\[
\sum \sigma (h_{(1)}, k_{(1)}) v(h_{(2)}k_{(2)}) =\sum v(h_{(1)})(h_{(2)}\cdot v(k_{(1)})) \omega (h_{(3)}, k_{(2)}) .
\]
Therefore
\[
\sum \sigma (h_{(1)}, k_{(1)}) v(h_{(2)}k_{(2)})u(h_{(3)} k_{(3)}) =\sum v(h_{(1)})(h_{(2)}\cdot v(k_{(1)})) \omega (h_{(3)}, k_{(2)}) u(h_{(4)} k_{(3)}).
\]
Remembering that the cocycle $\sigma$ has the absorption property
\[
\sigma (h,k) =\sum \sigma (h_{(1)} ,k_{(1)})(h_{(2)}k_{(2)} \bullet \um ),
\]
we obtain
\[
\sigma (h,k)=\sum v(h_{(1)})(h_{(2)}\cdot v(k_{(1)}))\omega (h_{(3)},k_{(2)}) u(h_{(4)}k_{(3)}) .
\]

($\Leftarrow$) Conversely, let us consider unit preserving maps $u,v\in \mbox{Hom}(H,A)$, satisfying the items (i) to (iv) in the statement. We shall  verify that $\Phi :A\#_{\omega} H \rightarrow A\#_{\sigma} H$ given by
\[
\Phi (a\#_{\omega } h)=\sum au(h_{(1)}) \#_{\sigma} h_{(2)}
\]
is indeed an algebra morphism.

 We see immediately that $\Phi (\um \#_{\omega} 1_H)=\um \#_{\sigma} 1_H$. For the multiplicativity, we have
\begin{eqnarray*}
& &\Phi (a\#_{\omega} h)\Phi (b\#_{\omega} k) =\\
&=& \sum (au(h_{(1)})\#_{\sigma}h_{(2)} )
(bu(k_{(1)})\#_{\sigma} k_{(2)}) \\
&=& \sum au(h_{(1)}) (h_{(2)} \bullet (bu(k_{(1)}))) \sigma (h_{(3)} ,k_{(2)})
\#_{\sigma} h_{(4)} k_{(3)} \\
&=& \sum au(h_{(1)}) v(h_{(2)})(h_{(3)} \cdot (bu(k_{(1)}))) u(h_{(4)}) v(h_{(5)})(h_{(6)} \cdot v(k_{(2)})) \times \\
& & \times \omega (h_{(7)},k_{(3)}) u(h_{(8)}k_{(4)}) \#_{\sigma} h_{(9)}k_{(5)} \\
&=& \sum a(h_{(1)}\cdot b) (h_{(2)}\cdot u(k_{(1)}))(h_{(3)}\cdot v(k_{(2)})) \omega (h_{(4)} , k_{(3)}) u(h_{(5)} k_{(4)})\#_{\sigma} h_{(6)}k_{(5)} \\
&=& \sum a(h_{(1)}\cdot b) (h_{(2)}\cdot (u(k_{(1)}) v(k_{(2)}))) \omega (h_{(3)} , k_{(3)}) u(h_{(4)} k_{(4)})\#_{\sigma} h_{(5)}k_{(5)} \\
&=& \sum a(h_{(1)}\cdot b) (h_{(2)}\cdot (k_{(1)}\cdot \um )) \omega (h_{(3)} , k_{(2)}) u(h_{(4)} k_{(3)})\#_{\sigma} h_{(5)}k_{(4)} \\
&=& \sum a(h_{(1)}\cdot b) \omega (h_{(2)} , k_{(1)}) u(h_{(3)} k_{(2)})\#_{\sigma} h_{(4)}k_{(3)} \\
&=& \Phi (\sum a(h_{(1)}\cdot b) \omega (h_{(2)} , k_{(1)})\#_{\omega} h_{(3)}k_{(2)} )\\
&=& \Phi ((a\#_{\omega} h)(b\#_{\omega} k)).
\end{eqnarray*}

Now, it remains to show that $\Phi$ is invertible. Consider the map $\Psi :A\#_{\sigma} H\rightarrow A\#_{\omega} H$ given by
\[
\Psi (a\#_{\sigma} h)=\sum av(h_{(1)})\#_{\omega} h_{(2)}.
\]
Then, we have
\[
\Psi (\Phi (a\#_{\omega} h))=\sum au(h_{(1)})v(h_{(2)})\#_{\omega} h_{(3)} =\sum a(h_{(1)}\cdot \um ) \#_{\omega} h_{(2)} =a\#_{\omega} h .
\]
From  (ii) and (iii), we easily conclude that $v*u (h)=h\bullet \um$, and then
\[
\Phi (\Psi (a\#_{\sigma} h))=\sum av(h_{(1)})u(h_{(2)})\#_{\sigma} h_{(3)} =\sum a(h_{(1)}\bullet \um ) \#_{\sigma} h_{(2)} =a\#_{\sigma} h .
\]
Therefore, $\Psi =\Phi^{-1}$ as we wanted to prove.
\end{proof}


\section{Partial Cleft Extensions}

It is a well-known simple fact that a group graded algebra ${\mathcal B}=\oplus_{g\in G} {\B}_g$ is isomorphic to a crossed product $\A \ast G,$ where ${\A}= {\B}_e $ and $e\in G$ is the neutral element of the group $G,$ exactly when  each ${\B}_g$ contains an element $u_g$ which is invertible in $\mathcal{B}.$ Evidently, the inverse $v_g$ of $u_g$ belongs to ${B}_{g\m}.$  Thus we have the maps $\gamma : G \to  \B,$
$g \mapsto u_g \in {\B}_g$ and ${\gamma}' : G \to { \B},$ $g \mapsto v_g\in {\B}_{g\m},$
and ${\gamma }'$ is in some sense inverse to $\gamma .$ This becomes  precise if we recall that  ${ \B}$ is a $\ka G$-module algebra, and a more general result for a Hopf algebra $H$ says that  an $H$-comodule algebra $B$ is isomorphic to a smash product $A \# H,$   $A = B^{co H},$ if and only if $A \subseteq B$ is a Cleft extension, which means that there exists a $ \ka $-linear map $\gamma : H \to B$ which fits into an appropriate  commutative diagram and  possesses a convolution inverse ${\gamma } ': H \to B.$

The partial case is essentially more complicated. One of the results in \cite{DES1}   gives a criteria for a non-degenerate  $G$-graded algebra ${\B}=\oplus_{g\in G} {\B}_g$ to have the structure of a  crossed product $\A \ast G$  by a twisted partial action of $G$ on  $\A= {\B}_e .$ More specifically,  if ${\B} $ satisfies
\begin{equation}
\label{g-graduada}
{\B}_g {\B}_{g^{-1}} {\B}_g ={\B}_g,   \;\;\;\;\;\;\;\;\;\;\;\; (\forall g \in G),
\end{equation}
then using the multiplication in ${\B}$ it is possible to define, for each $g \in G$,   idempotent ideals $\mathcal{D}_g =\mathcal{B}_g \mathcal{B}_{g^{-1}},$  $\mathcal{D}_{g^{-1}} =\mathcal{B}_{g^{-1}} \mathcal{B}_g$ of $\mathcal{B}_e$, a unital  $\mathcal{D}_g$ - $\mathcal{D}_{g^{-1}}$ bimodule $\mathcal{B}_g$ and a unital  $\mathcal{D}_{g^{-1}}$ - $\mathcal{D}_g$ bimodule $\mathcal{B}_{g^{-1}},$ such  that they constitute a Morita context. The main ingredients used to the construction of the crossed product are operators $u_g$ and $v_g$ in the multiplier algebra of the context algebra
\[
\mathcal{C}_g =\left( \begin{array}{cc} \mathcal{D}_g & \mathcal{B}_g \\
\mathcal{B}_{g^{-1}} & \mathcal{D}_{g^{-1}} \end{array} \right)
\]
such that
\begin{equation}
\label{g-graduada2}
u_g v_g =e_{11}= \left( \begin{array}{cc} 1 & 0 \\ 0 & 0 \end{array} \right) \qquad
\mbox{and} \qquad
v_g u_g =e_{22}= \left( \begin{array}{cc} 0 & 0 \\ 0 & 1 \end{array} \right).
\end{equation} Then it turns out that a non-degenerate $G$-graded algebra $\mathcal{B}$ is isomorphic as a graded algebra to the crossed product $\A \ast G$ by a twisted partial action exactly when  (\ref{g-graduada}) is satisfied and for each $g\in G$ there exist multipliers $u_g$ and $v_g$ of $\mathcal{C}_g$ such that (\ref{g-graduada2}) holds.\\

Now it becomes natural to treat this topic in the context of twisted partial Hopf actions, which is the purpose of the present section. The ``partiality'' is reflected now on the properties of $\gamma .$ Instead of assuming that  $\gamma$ is convolution invertible, one  declares the existence of  a map  ${\gamma }'$ which is related to $\gamma $ by conditions which are weaker than that of the convolution invertibility. Some of them  match  equalities which already played a crucial role in the study of partial actions and partial representations (see Remark~\ref{interaction}).

\begin{defi}
\label{cleft}Let $B$ be a right  $H$-comodule unital algebra with coaction
 given by $\rho: B \rightarrow B \otimes H$ and let $A$ be a subalgebra of $B$. We will say that $A \subset B$ is an $H$-extension if $A = B^{co H}$. An $H$-extension $A \subset B$ is \emph{partially cleft} if there is a pair of $k$-linear maps $\gamma, \gi : H \rightarrow B$ such that
\begin{enumerate}[(i)]
\item $\gamma (1_H) = 1_B ,$
\item the diagrams below are commutative:
\beqn \label{diagrams}
\xymatrix{
H \ar[r]^{\gamma} \ar[d]^{\Delta}& B \ar[d]^{\rho}
& & H \ar[r]^{\gamma'} \ar[d]^{\Delta^{cop}}& B \ar[d]^{\rho} \\
H \ot H \ar[r]_{\gamma \ot I_H} & B \ot H
& & H \ot H \ar[r]_{\gamma' \ot S} & B \ot H
}
\eqn
\item $(\gamma*\gamma')\circ M$ is a central element in the convolution algebra $\Hom(H \otimes H,A),$ where $M: H \otimes H \to H$ is the multiplication in $H ,$  and   $(\gamma' * \gamma ) (h)$ commutes with every element of $A$ for each $h\in H,$
\end{enumerate}
and, for all $b \in B$ and $h,k \in H$, if we write $e_h= ({\gm } \ast {\gm}')(h)$ and $\tilde{e}_h =  ({\gm}' \ast \gm ) (h)$, then
\begin{enumerate}
\item[(iv)] $\sum b_{(0)} \gi (b_{(1)}) \gamma (b_{(2)}) = b ,$ \label{expansao}
\item[(v)] $ {\gm}(h) e_k  = \sum  e_{h_{(1)}k} {\gm} (h_{(2)}), $
\item[(vi)] $   {\gm}'(k) \tilde{e}_h = \sum \tilde{e}_{hk_{(1)}}{\gm}'(k_{(2)}),$
\item[(vii)] $  \sum \gm (h k_{(1)}) \tilde{e}_{k_{(2)}} = \sum {e}_{h_{(1)}} \gm (h_{(2)} k),$
\end{enumerate}

\end{defi}

Note  that item (iii) makes sense because item (ii) implies that $(\gamma * \gamma')(h) \in A$, for all $h \in H$, and therefore $\gamma * \gamma' \in \Hom(H,A)$.

With respect to items (v), (vi) and (vii) we make the following:

\vu

\begin{remark}\label{interaction} Let $\gamma : G \to \B$ be a partial representation of a group $G$ into a $\ka $-algebra $\B,$ i.e. a $\ka $-linear map such that $\gamma (1_G) = 1_B ,$ ${\gm}(g) {\gm}(s) \gm (s\m)= {\gm}(g s) \gm (s\m)$
and ${\gm}(g\m ) {\gm}(g) \gm (s) = {\gm}(g\m ) {\gm}(gs),$ for all $g,s \in G .$ Then by (2)  of \cite{DEP} the following equality holds
$$\gamma (g) e_r = e_{gr} \gm (g) ,$$ where $e_g =\gm (g) {\gm}(g\m ).$  This corresponds to item (v) if we take $H= \ka G.$ The above  equality plays a crucial role for the interaction between  partial actions  and  partial representations (see \cite{DE}), as well as for   an  analogous interaction  in the context of partial projective representations (see \cite{DN}, \cite{DN2}).   Now writing ${\gm}'(g) = \gm (g\m )$ and $\tilde{e}_g = e_{g\m},$ we readily obtain from the above equality that
$$      {\gm}'(g) \tilde{e}_s = \tilde{e}_{sg}{\gm}'(g)\quad \text{and}  \quad  \gm (g s) \tilde{e}_{s} =  {e}_{g} \gm (gs),$$ for all $g,s \in G,$ which are exactly items (vi) and (vii) above with $H = \ka G.$
\end{remark}

Note that in the case of  a cleft extension, with a convolution invertible map $\gamma$,  the axioms for partial cleft extensions are automatically satisfied if we take ${\gm}'$ to be the convolution inverse of $\gm .$

Observe furthermore that given a partial cleft extension, we also have
\begin{equation}\label{gammalinhaUm}
\gamma'(1_H) = 1_B ,
\end{equation} since  by (iv) of Definition \ref{cleft} we see that
\[
1_B = \sum (1_B)_{(0)} \; \gamma'((1_B ) _{(1)}) \; \gamma((1_B)_{(2)}) = (1_B) \gamma'(1_H) \gamma(1_H) = \gamma'(1_H).
\] Moreover, since  by   (\ref{diagrams}) $\gamma$ is a morphism of comodules,  we have that $\rho^2 (\gamma (h)) = (\text{I}_A\ot\Delta)\rho(\gamma(h))=\sum \gamma (h_{(1)} ) \otimes h_{(2)} \otimes h_{(3)} .$ Then   applying  (iv) of Definition \ref{cleft} to $b = \gamma(h),$ we conclude that
\begin{equation}
\label{produtogama}
 \gamma * \gamma' * \gamma = \gamma .
\end{equation}
The latter will be quite important in what follows. In particular, multiplying this equality by ${\gamma }'$ on the right we obtain that $\gamma \ast {\gamma}'$ is an idempotent, and, moreover, multiplying (\ref{produtogama})  by $\gamma'$ on the left, we see that $\gamma' * \gamma$ is also  idempotent.  Furthermore, since any linear function $\tau \in \Hom (H,A)$ can be seen as a function $ h \otimes k \mapsto \tau (h)$ in $\Hom (H\otimes H,A), $  item (iii) of Definition~\ref{cleft} implies that
$(\gamma \ast {\gamma}') \ast \tau = \tau \ast   (\gamma \ast {\gamma}'),$ so that we have:

\begin{remark}\label{rem:central}  Given a partially cleft extension, $ \gamma*\gamma' $ is a central idempotent in the convolution algebra $\Hom(H ,A).$
\end{remark}

The map $\gamma'$ may not satisfy an equality similar to (\ref{produtogama}), but it always can be replaced by another map $\gb$ that does, and the pair $(\gamma,\gb)$ still satisfies properties (i)--(vii), as seen in the following:

\begin{lemma}  We may assume that ${\gm }'$ in Definition~\ref{cleft} satisfies the equality
\begin{equation}
\label{produtogamalinha}
\gamma' * \gamma * \gamma' = \gamma'.
\end{equation}
\end{lemma}

\begin{proof} Consider the map $\gb = \gamma' * \gamma * \gamma'$. Since $\gamma'*\gamma$ is an idempotent,
\begin{eqnarray*}
 \gb * \gamma * \gb & = & (\gamma' * \gamma * \gamma') * \gamma * (\gamma' * \gamma * \gamma') \\
& = & (\gamma' * \gamma) * (\gamma' * \gamma) * (\gamma' * \gamma) * \gamma' \\
 & = & \gamma' * \gamma * \gamma' = \gb.
\end{eqnarray*}
We will show that the pair $(\gamma,\gb)$ satisfies the properties (i)--(vii).  Item (i) is immediate in view of (\ref{gammalinhaUm}). Item (ii) holds, since
\begin{eqnarray*}
\rho(\gb(h)) & = & \sum \rho (\gamma'(h_{(1)})) \gamma(h_{(2)}) \gamma'(h_{(3)})) = \sum \rho  (\gamma'(h_{(1)}) ) \rho ( \gamma(h_{(2)}) ) \rho ( \gamma'  (h_{(3)}) ) \\
& = & \sum (\gamma'(h_{(2)}) \ot S(h_{(1)}))(\gamma(h_{(3)}) \ot h_{(4)}) ( \gamma'(h_{(6)}) \ot S(h_{(5)})) 
\end{eqnarray*}
\begin{eqnarray*}
& = & \sum \gamma'(h_{(2)})\gamma(h_{(3)})\gamma'(h_{(6)}) \ot S(h_{(1)})h_{(4)} S(h_{(5)}) \\
& = & \sum \gamma'(h_{(2)})\gamma(h_{(3)})\gamma'(h_{(4)}) \ot S(h_{(1)}) \\
& = & \sum \gb(h_{(2)})\ot S(h_{(1)}) =  (\gb \ot S)\Delta^{cop} (h).
\end{eqnarray*} Item (iii) immediately follows from
\begin{equation}\label{barra}
\gamma * \gb = \gamma * \gamma', \ \ \gb * \gamma = \gamma' * \gamma.
\end{equation} Item (iv) holds because, given $b \in B$,
\begin{eqnarray}
& & \sum b_{(0)} \gb(b_{(1)}) \gamma(b_{(2)}) = \sum b_{(0)} \gamma'(b_{(1)}) (\gamma(b_{(2)}) \gamma'(b_{(3)}) \gamma(b_{(4)})) \nonumber \\
& & = \sum b_{(0)} \gamma'(b_{(1)}) \gamma(b_{(2)}) = b. \nonumber
\end{eqnarray} It remains to check  (v)--(vii). Notice that  ${e}_h=  (\gm \ast \bar{\gm}) (h)$ and  ${\tilde{e}}_h = (\bar{\gm} \ast \gm) (h),$ thanks to (\ref{barra}).  Thus for $\bar{\gm}$ we need to verify  only  (vi). For compute
\begin{align*} \bar{\gm} (k) \tilde{e}_h = {\gm}'(k_{(1)}) e_{k_{(2)}} \tilde{e}_h \overset{\text{(iii)}}{=} {\gm}'(k_{(1)})  \tilde{e}_h e_{k_{(2)}}  \overset{\text{(vi)}}{=} \tilde{e}_{ hk_{(1)}}   \, {\gm}'(k_{(2)})  e_{k_{(3)}} =
\tilde{e}_{hk_{(1)} }  \, \bar{\gm}(k_{(2)}),
\end{align*} taking into account that $e_{k} \in A .$ \end{proof}




 Since $\rho$ is an algebra morphism, applying  (iv) of Definition~\ref{cleft}  to $b=\gamma(h)  \gamma (k)$ and also to $b=\gamma(h)  \gamma (k)a ,$ we obtain for any  $a \in A = B^{coH}$ and $h,k \in H$ the following equalities:


\begin{eqnarray}
\gamma(h)  \gamma (k)  & = & \sum \gamma(h_{(1)})  \gamma (k_{(1)}) \gamma'(h_{(2)}k_{(2)}) \gamma (h_{(3)}k_{(3)}), \label{gammagamma}\\
\gamma(h) \gamma (k) a & = & \sum  \gamma(h_{(1)})  \gamma (k_{(1)}) a \gamma'(h_{(2)}k_{(2)}) \gamma (h_{(3)}k_{(3)}) . \label{gammadir}
\end{eqnarray}
Then taking $k=1_H$ in (\ref{gammadir}) we have
\begin{eqnarray}
 \gamma(h)  a   &=&  \sum \gamma(h_{(1)})  a  \gamma'(h_{(2)}) \gamma (h_{(3)}). \label{gammavezesa}
\end{eqnarray}


\begin{prop}\label{CleftProp} If $(A,\cdot, (\omega,\omega'))$ is a symmetric partial twisted $H$-module algebra, then $A \subset \AH$ is a partially cleft $H$-extension.
\end{prop}

\begin{proof} We see that $\AH$ is a right comodule algebra via the mapping $\rho = (I \ot \Delta): \AH \rightarrow ( \AH) \ot H$. It is easy to see that $(\AH)^{coH} = A \ot 1_H$, which we will identify with $A$ via the canonical monomorphism $A \rightarrow A \ot 1_H$.

Consider the maps $\gamma, \gamma': H \rightarrow \AH$ given by
\begin{align}
 \gamma(h) & = \um \# h = (\um \otimes h)(\um \otimes 1_H), \\
\gamma'(h) & = \sum \omega'(S(h_{(2)})  , h_{(3)}) \# S(h_{(1)}).
\end{align}

From the definition of $\gamma$ we have $\gamma(1_H) = \um \# 1_H =  1_{\AH },$ which gives (i) of Definition~\ref{cleft}. With respect to item (ii), the equality $\rho \gamma = (\gamma \ot I) \Delta$ follows directly by the definition of $\rho$. As for the second diagram in (ii), we have:

\begin{eqnarray*}
\rho \gamma' (h) & = &\sum \rho (\omega'(S(h_{(2)}) , h_{(3)}) \# S(h_{(1)})) = \sum  (\omega'(S(h_{(3)}) , h_{(4)}) \# S(h_{(2)}) \ot S(h_{(1)}) \\
& = & \sum  \gamma'(h_{(2)}) \ot S(h_{(1)}) = ( \gamma' \ot S) \Delta^{cop}(h),
\end{eqnarray*} which completes  the proof of  (ii) of the definition of partial cleft extension. Now,
\begin{eqnarray*}
(\gamma * \gamma')(h)
& = & \sum (\um \# h_{(1)})(\omega'(S(h_{(3)}) , h_{(4)}) \# S(h_{(2)})) \\
& = & \sum  (h_{(1)} \cdot \omega'(S(h_{(6)}) , h_{(7)})) \omega(h_{(2)} , S(h_{(5)})) \# \underbrace{ h_{(3)}S(h_{(4)}) } \\
& = & \sum  (h_{(1)} \cdot \omega'(S(h_{(4)}) , h_{(5)})) \omega(h_{(2)} , S(h_{(3)})) \# 1_H \\
& \overset{\text{(\ref{h.em.omega.linha})}}{=} & \sum  \omega(h_{(1)} , \underbrace{S(h_{(8)})h_{(9)}}) \omega'(h_{(2)} S(h_{(7)}) , h_{(10)})
 \times \\
 && \times \underbrace{\omega'(h_{(3)} , S(h_{(6)})) \omega(h_{(4)} , S(h_{(5)}))} \# 1_H\\
 &\overset{\text{(\ref{omega.omegalinha}),(\ref{8})}}{=} & \sum
 (h_{(1)} \cdot \um) \underbrace{ \omega'(h_{(2)} S(h_{(6)}) , h_{(7)})(h_{(3)} \cdot \um) }  \overbrace{ (h_{(4)}S(h_{(5)}) } \cdot \um) \# 1_H \\
& \overset{\text{(\ref{abs.omegalinha})}}{=}  & \sum  (h_{(1)} \cdot \um) \omega'( \underbrace{h_{(2)} S(h_{(3)} ) } , h_{(4)})  \# 1_H = (h \cdot \um) \# 1_H.
\end{eqnarray*} Hence $(\gamma * \gamma')(hk) = f_2(h,k)   \# 1_H,  $ and  this implies that $(\gamma * \gamma')\circ M$ is central in $\Hom (H\otimes H, A)$ thanks to the convolution
centrality of $f_2 .$  Observe also that $(\gamma * \gamma')(h)  $   commutes with every element of $A ,$ since  each $a \in A$ gives rise to a linear map $\tau_a: H \rightarrow A$ defined by $ \tau_a(h) = \varepsilon(h) a$, and  ${\bf e}(h) = (h \cdot \um)$ is central in $\Hom(H,A)$ by assumption. Hence
\[
(h \cdot \um)a = \sum (h_{(1)} \cdot \um)\varepsilon(h_{(2)})a = ({\bf e}*\tau_a)(h) = (\tau_a * {\bf e})(h) = a (h \cdot \um).
\]
With respect to $\gamma'*\gamma$,
\beqnast
\gamma' * \gamma (h ) &=&(\omega'(S(h_{(2)}) , h_{(3)}) \# S(h_{(1)}))(\um \# h_{(4)}) \\
&=&\omega' (S(h_{(4)}) , h_{(5)}) (S(h_{(3)}) \cdot \um)\omega (S(h_{(2)}) , h_{(6)})\# S(h_{(1)})h_{(7)}  \\
&=&\omega' (S(h_{(3)}) , h_{(4)})\omega (S(h_{(2)}) , h_{(5)})\# S(h_{(1)})h_{(6)}  \\
&=&(S(h_{(3)})h_{(4)} \cdot \um )(S(h_{(2)}) \cdot \um)\# S(h_{(1)})h_{(5)}  \\
&=&(S(h_{(2)}) \cdot \um)\# S(h_{(1)})h_{(3)},
\eqnast
and  this expression implies
\beqnast
(\gamma' * \gamma )(h )(a \# 1_H) &= &  \sum ((S(h_{(2)}) \cdot \um)\# S(h_{(1)})h_{(3)})(a \# 1_H)  \\
& = & \sum (S(h_{(4)}) \cdot \um)(S(h_{(3)})h_{(5)} \cdot a) \omega (S(h_{(2)})h_{(6)} , 1_H) \# S(h_{(1)})h_{(7)} \\
& = & \sum\underbrace{ (S(h_{(4)})h_{(5)} \cdot a)(S(h_{(3)}) \cdot \um)} \omega (S(h_{(2)})h_{(6)} , 1_H) \# S(h_{(1)})h_{(7)} \\
& = & \sum a (S(h_{(3)}) \cdot \um) \omega(S(h_{(2)})h_{(4)}, 1_H) \# S(h_{(1)})h_{(5)}\\
& = & \sum a \omega(\underbrace{S(h_{(3)})h_{(4)}}, 1_H)(S(h_{(2)}) \cdot \um)  \# S(h_{(1)})h_{(5)}\\
& = & \sum a (S(h_{(2)}) \cdot \um)  \# S(h_{(1)})h_{(3)}\\
& = & \sum (a \# 1_H)( (S(h_{(2)}) \cdot \um)  \# S(h_{(1)})h_{(3)}) =  (a \# 1_H)(\gamma' * \gamma )(h ),
\eqnast proving item (iii).

For item (iv), consider $b = a \# h$ in $\AH$. Applying $\rho^2=(\text{I}_A\ot\Delta)\rho$ to $b$ we obtain
\[
\sum b_{(0)} \ot b_{(1)} \ot b_{(2)} = \sum (a \# h_{(1)}) \ot h_{(2)} \ot h_{(3)}, \] and therefore
\begin{eqnarray*}
& & \sum b_{(0)} \gamma'(b_{(1)})\gamma(b_{(2)}) = \sum  (a \# h_{(1)}) \gamma'(h_{(2)}) \gamma(h_{(3)}) \\
& = & \sum  (a \# h_{(1)}) (\omega'(S(h_{(3)}) \ot h_{(4)}) \# S(h_{(2)}))(\um \# h_{(5)}) \\
& = & \sum  (a \# 1_H)\underbrace{(\um\#h_{(1)}) (\omega'(S(h_{(3)}) \ot h_{(4)}) \# S(h_{(2)}))}  (\um \# h_{(5)}) \\
& \overset{(\ref{gammadir}),(\ref{gammavezesa})}{=}  & \sum  (a \# 1_H)(\gamma * \gamma')(h_{(1)})(\um \# h_{(2)})\\
 &=& \sum  (a \# 1_H)((h_{(1)} \cdot \um) \# 1_H)(\um \# h_{(2)}) =  a \# h = b.
 \end{eqnarray*}

Next we check (v), using (iii) of Definition~\ref{symm}, as follows:
\begin{align*}   \sum e_{h_{(1)}k}  {\gm} (h_{(2)}) &= \sum ( h_{(1)}k \cdot \um \# 1_H ) (\um \# h_{(2)})
\overset{(\ref{gammadir})}{=} \sum (h_{(1)} k \cdot \um ) (h_{(2)} \cdot \um ) \# h_{(3)}\\
&=   \sum (h_{(1)} \cdot (k \cdot \um ))  \# h_{(2)} = \sum (h_{(1)} \cdot (k \cdot \um )) (h_{(2)}\cdot \um )  \# h_{(3)}\\
& = (\um \# h)  (k \cdot \um \# 1_H) = \gm(h)(\gm\ast\gm')(k)={\gm}(h) e_{k}.
\end{align*}

In order to establish (vi) we compute, using again (iii) of Definition~\ref{symm}, that
\begin{align*} & {\gm}'(h) \tilde{e} _k = (\sum {\omega}' (S(h_{(2)}), h_{(3)}) \# S(h_{(1)}) \; (\sum S(k_{(2)}) \cdot \um \# S(k_{(1)})k_{(3)}) =\\
&= \sum {\omega}'(S(h_{(4)}), h_{(5)}) \underbrace{(S(h_{(3)}) \cdot ( S(k_{(3)}) \cdot \um ))} \omega ( S(h_{(2)}), S(k_{(2)}) k_{(4)}) \# S(h_{(1)}) S(k_{(1)}) k_{(5)}\\
&    =\sum \underbrace{{\omega}'(S(h_{(5)}), h_{(6)}) (S(h_{(4)}) \cdot \um) }(S(h_{(3)})  S(k_{(3)}) \cdot \um )\omega ( S(h_{(2)}), S(k_{(2)}) k_{(4)}) \#\\ & \# S(h_{(1)}) S(k_{(1)}) k_{(5)}\overset{(\ref{abs.omegalinha})}{=}  \\
\overset{(\ref{abs.omegalinha})}{=} &\sum {\omega}'(S(h_{(4)}), h_{(5)})  \underbrace{(S(h_{(3)}) S(k_{(3)}) \cdot \um )\omega ( S(h_{(2)}), S(k_{(2)}) k_{(4)}) } \#  S(h_{(1)}) S(k_{(1)}) k_{(5)}.
\end{align*} 
With respect to the underbraced product, for a fixed $m\in H$ consider the function  $\tau _m :H\ot H \to A$ given by
$h\ot k \mapsto \omega ( h, k m) .$ Since $f_2$ is central, we have
\begin{align*}
&  \sum (S(h_{(2)}) S(k_{(2)}) \cdot \um )\omega ( S(h_{(1)}), S(k_{(1)}) m) = ( f_2 \ast \tau _m)(S(h)\ot S(k ) ) =\\
& = (\tau _m \ast f_2)(S(h)\ot S(k ) ) =
\sum \omega ( S(h_{(2)}), S(k_{(2)}) m) (S(h_{(1)}) S(k_{(1)}) \cdot \um ),
\end{align*} and consequently we obtain
\begin{align*} & {\gm}'(h) \tilde{e} _k =\\
& = \sum {\omega}'(S(h_{(4)}), h_{(5)})  \omega ( S(h_{(3)}), \underbrace{ S(k_{(3)}) k_{(4)} })  (S(h_{(2)}) S(k_{(2)}) \cdot \um ) \#  S(h_{(1)}) S(k_{(1)}) k_{(5)} \\
& = \sum \underbrace{  {\omega}'(S(h_{(4)}), h_{(5)})  ( S(h_{(3)}) \cdot \um )} \;  (S(h_{(2)}) S(k_{(2)}) \cdot \um ) \#  S(h_{(1)}) S(k_{(1)}) k_{(3)}\\
&\overset{(\ref{abs.omegalinha})}{=}  \sum   {\omega}'(S(h_{(3)}), h_{(4)})   \;  (S(h_{(2)}) S(k_{(2)}) \cdot \um ) \#  S(h_{(1)}) S(k_{(1)}) k_{(3)}.
\end{align*} To compute $\sum \tilde{e}_{kh_{(1)}} {\gm}'(h_{(2)})$ consider first the function
$\mu _{l, m, n} : H\to A$ defined by $h \mapsto hl \cdot {\omega}'(m,n),$ where $l, m, n \in H$ are fixed.   Then
${\bf e} \ast  \mu _{l, m, n} = \mu _{l, m, n} \ast {\bf e},$ and applying both sides of this equality to $S( kh ),$ we obtain
 \begin{equation}\label{commuting1} (S(k_{(2)} h_{(2)}) \cdot \um ) \; [(S(k_{(1)} h_{(1)})l) \cdot {\omega }'(m , n)]  = [ (S(k_{(2)}h_{(2)})l) \cdot {\omega }'(m, n) ]\;  (S(k_{(1)}h_{(1)}) \cdot \um ).
\end{equation}  Similarly, taking the function $\nu _{m,n} : H \to A,$ given by $h \mapsto \omega (hm, n),$ and using
${\bf e} \ast  \nu _{m,n} = \nu _{m,n} \ast {\bf e}$ applied also to $S(hk),$ we also obtain
 \begin{equation}\label{commuting2} (S(k_{(2)}h_{(2)} ) \cdot \um ) \;  {\omega }( S(k_{(1)}h_{(1)} ) m , n)  =
\omega ( S(k_{(2)}h_{(2)} ) m, n) \;  ( S(k_{(1)}h_{(1)} ) \cdot \um ).
\end{equation}  Then we have:
\begin{align*} & \sum \tilde{e}_{kh_{(1)}} {\gm}'(h_{(2)}) =\\
&  [ \sum ( S(k_{(2)}h_{(2)}) \cdot \um  ) \# S(k _{(1)}h _{(2)}) k_{(3)}h_{(3)} ]  \; \;  [ \sum {\omega}'(S(h_{(5)}), h_{(6)})  ) \# S(h_{(4)}) ]=\\
&   \sum ( S(k_{(4)}h_{(4)}) \cdot \um  )  \;\; [ S(k _{(3)}h _{(3)}) k_{(5)}h_{(5)}  \cdot    {\omega}'(S(h_{(10)}), h_{(11 )})  ) ] \times \\
& \times \omega ( S(k _{(2)}h _{(2)}) k_{(6)}h_{(6)}, S(h_{(9)}) )  \#  S(k _{(1)}h _{(1)}) k_{(7)} \underbrace{ h_{(7)} S(h_{(8)}) }=\\
&   \sum ( S(k_{(4)}h_{(4)}) \cdot \um  )  \;\; [ S(k _{(3)}h _{(3)}) k_{(5)}h_{(5)}  \cdot     {\omega}'(S(h_{(8)}), h_{(9 )})  )]   \times \\
&\times \omega ( S(k _{(2)}h _{(2)}) k_{(6)}h_{(6)}, S(h_{(7)}) )   S(k _{(1)}h _{(1)}) k_{(7)} \overset{\text{(\ref{commuting1})}}{=} \\
&   \sum  [  \underbrace{ S(k _{(4)}h _{(4)}) k_{(5)}h_{(5)} }  \cdot    {\omega}'(S(h_{(8)}), h_{(9 )})  ) ] \;\;
( S(k_{(3)}h_{(3)}) \cdot \um  )  \times \\
& \times \omega ( S(k _{(2)}h _{(2)}) k_{(6)}h_{(6)}, S(h_{(7)}) )  \#  S(k _{(1)}h _{(1)}) k_{(7)} =\\
&   \sum    {\omega}' (S(h_{(6)}), h_{(7 )})  )  \underbrace{ ( S(k_{(3)}h_{(3)}) \cdot \um  )  \omega ( S(k _{(2)}h _{(2)}) k_{(4)}h_{(4)}, S(h_{(5)}) ) } \#   S(k _{(1)}h _{(1)}) k_{(5)} =\\
 & \sum    {\omega}' (S(h_{(6)}), h_{(7 )})  )    \omega ( \underbrace{ S(k _{(3)}h _{(3)}) k_{(4)}h_{(4)} }, S(h_{(5)}) )  ( S(k_{(2)}h_{(2)}) \cdot \um  )  \#    S(k _{(1)}h _{(1)}) k_{(5)} =\\
 & \sum   \underbrace{  {\omega}' (S(h_{(4)}), h_{(5 )})  )   (S(h_{(3)})  \cdot \um ) } ( S(k_{(2)}h_{(2)}) \cdot \um  )  \#
S(k _{(1)}h _{(1)}) k_{(3)} =\\
 & \sum     {\omega}' (S(h_{(3)}), h_{(4 )})  )   ( S(h_{(2)}) S(k_{(2)}) \cdot \um  )  \#    S(h _{(1)}) S(k _{(1)}) k_{(3)},\\
\end{align*}which coincides with the expression obtained above for $ {\gm}'(h) \tilde{e} _k ,$ proving thus (vi).

\vd

 Finally, item (vii) follows from the next calculation, in which we use again the convolution centrality of $f_2$ and ${\bf e}$:
\begin{align*} 
& \sum \gamma (hk_{(1)} ) \tilde{e}_{k_{(2)}} = \sum (\um \# h k_{(1)} )  \; [ ( S(k_3) \cdot \um ) \# S(k_{(2)} k_{(4)} )] =\\
& \sum (h_{(1)} k_{(1)} \cdot ( S(k_{(6)})  \cdot \um )) \; \omega (  h_{(2)} k_{(2)} ,   S(k_{(5)})  k_{(7)} ) \#
h_{(3)} \underbrace{ k_{(3)}  S(k_{(4)}) } k_{(8)} = \\
& \sum (h_{(1)} k_{(1)} \cdot ( S(k_{(4)})  \cdot \um )) \; \omega (  h_{(2)} k_{(2)} ,   S(k_{(3)})  k_{(5)} ) \#
h_{(3)}  k_{(6)} = 
\\
& \sum \underbrace{ (h_{(1)} k_{(1)} \cdot \um) (  (h_{(2)} k_{(2)} S(k_{(5)})  \cdot \um )} \; \omega (  h_{(3)} k_{(3)} ,   S(k_{(4)})  k_{(6)} ) \#
h_{(4)}  k_{(7)} = \\
& \sum   (h_{(1)} k_{(1)} S(k_{(5)})  \cdot \um ) (h_{(2)} k_{(2)} \cdot \um)  \; \omega (  h_{(3)} k_{(3)} ,   S(k_{(4)})  k_{(6)} ) \#
h_{(4)}  k_{(7)} = \\
& \sum   (h_{(1)} k_{(1)} S(k_{(5)})  \cdot \um ) \underbrace{ (h_{(2)} k_{(2)} \cdot \um)  \omega (  h_{(3)} k_{(3)} ,   S(k_{(4)})  k_{(6)} )} \#
h_{(4)}  k_{(7)} = 
\end{align*}
\begin{align*} 
& \sum   (\underbrace{ h_{(1)} k_{(1)} S(k_{(4)})  \cdot \um )   \omega (  h_{(2)} k_{(2)} ,   S(k_{(3)})  k_{(5)} ) } \#
h_{(3)}  k_{(6)} =\\
& \sum     \omega (  h_{(1)} k_{(1)} ,   \underbrace{ S(k_{(4)})  k_{(5)} } ) \;  (h_{(2)} \underbrace{ k_{(2)} S(k_{(3)} } )  \cdot \um )  \#
h_{(3)}  k_{(6)} =\\
 & \sum     (  h_{(1)} k_{(1)}  \cdot \um  )  (h_{(2)}  \cdot \um  )    \#    h_{(3)}  k_{(2)} =
\sum   (h_{(1)}  \cdot \um  ) \;   (  h_{(2)} k_{(1)}  \cdot \um  )     \#   h_{(3)}  k_{(2)} =\\
& \sum   (h_{(1)}  \cdot \um  ) \;   \omega ( 1_H,  h_{(2)} k_{(1)}    )     \#   h_{(3)}  k_{(2)}
\overset{(\ref{firstprop})}{=}
 \sum (  (h_{(1)} \cdot \um ) \# 1_H  ) \; ( \um \# h_{(2)} k ) =\\
& \sum e_{h_{(1)}} \gamma (h_{(2)} k ).
\end{align*}   \end{proof}

ˆ


\begin{thm}\label{the51}
Let $B$ be an $H$-comodule algebra and let $A = B^{coH}$. Then the $H$-extension $A \subset B$ is partially cleft if and only if $B$ is isomorphic to a partial crossed product $\AH$ with respect to a symmetric twisted partial $H$-module structure on $A.$
\end{thm}

\begin{proof}
We have already proved half of this statement in  Proposition~\ref{CleftProp}. So, assume  that $B$ is partially  cleft by the pair of maps $\gamma, \gamma' : H \rightarrow B$.
The pair $(\gamma, \gamma')$ allows us to define a twisted partial action of $H$ on $A = B^{coH}$ as follows. Given $h,k \in H$ and $a \in A$,  set
\begin{eqnarray*}
h \cdot a & = & \sum \gamma(h_{(1)}) a \gamma'(h_{(2)}),\\
\omega(h ,k) & = & \sum \gamma(h_{(1)}) \gamma(k_{(1)}) \gamma'(h_{(2)}k_{(2)}),\\
\omega'(h ,k) & = & \sum \gamma(h_{(1)}k_{(1)}) \gamma'(k_{(2)}) \gamma'(h_{(2)}).\\
\end{eqnarray*}
Before anything else, we must check that these elements lie in $A$, but this is quite simple.
\begin{eqnarray*}
\rho(h \cdot a) & = & \sum \rho(\gamma(h_{(1)}))\rho( a) \rho( \gamma'(h_{(2)})) \\
& = & \sum (\gamma(h_{(1)}) \ot h_{(2)})( a \ot 1_H) (\gamma'(h_{(4)}) \ot S(h_{(3)})) \\
& = & \sum (\gamma(h_{(1)})a\gamma'(h_{(4)}) \ot h_{(2)}S(h_{(3)}) \\
& = & \sum (\gamma(h_{(1)})a\gamma'(h_{(2)}) \ot 1_H \\
&= &(h \cdot a) \ot 1_H, \\
\end{eqnarray*} and thus $a\in B^{coH}= A.$
In an analogous fashion, one may check that both $\omega(h ,k)$ and $\omega'(h , k)$ lie in $A$ for every $h,k$ in $H$. For instance,
\begin{eqnarray*}
\rho(\omega(h,k)) & = & \sum (\gamma(h_{(1)}) \ot h_{(2)}) (\gamma(k_{(1)}) \ot k_{(2)}) (\gamma'(h_{(4)}k_{(4)}) \ot S(h_{(3)}k_{(3)}))\\
& = & \sum \gamma(h_{(1)}) \gamma(k_{(1)}) \gamma'(h_{(2)}k_{(2)}) \ot 1_H\\
& = & \omega(h,k) \ot 1_H,
\end{eqnarray*} and similarly for $\omega'(h , k).$

Note that, since $A=B^{coH}$, then $1_B =\um $, and since $\gamma(1_H) = 1_B =\um = \gamma'(1_H)$, we have $1_H \cdot a = a$ for all $a \in A$.

Next, given $h \in H$ and $a,b \in A$,  we see  that
\begin{eqnarray*}
h \cdot ab & = &  \sum \underbrace{\gamma(h_{(1)}) a} b\gamma'(h_{(2)}) \\
& \overset{\text{(\ref{gammavezesa})}}{=}  & \sum \gamma(h_{(1)}) a \gamma'(h_{(2)}) \gamma(h_{(3)}) b\gamma'(h_{(4)})\\
& =& \sum (h_{(1)} \cdot a) (h_{(2)} \cdot b).
\end{eqnarray*}

The partial action is twisted by $(\omega, \omega')$, since
\begin{eqnarray*}
h \cdot (k \cdot a)
& =& \sum \underbrace{\gamma(h_{(1)}) \gamma(k_{(1)}) a} \gamma'(k_{(2)}) \gamma'(h_{(2)}) \\
& \overset{\text{(\ref{gammadir})}}{=}& \sum \underbrace{\gamma(h_{(1)})\gamma(k_{(1)})} a \gamma'(h_{(2)}k_{(2)}) \gamma(h_{(3)}k_{(3)})  \gamma'(k_{(4)})\gamma'(h_{(4)})\\
& \overset{\text{(\ref{gammagamma})}}{=}& \sum [\gamma(h_{(1)})\gamma(k_{(1)}) \gamma'(h_{(2)}k_{(2)})] \gamma(h_{(3)}k_{(3)}) a \gamma'(h_{(4)}k_{(4)}) \times \\
& & \times [\gamma(h_{(5)}k_{(5)})  \gamma'(k_{(6)})\gamma'(h_{(6)})]\\
& =& \sum \omega(h_{(1)} , k_{(1)})(h_{(2)}k_{(2)} \cdot  a)\omega'(h_{(3)} , k_{(3)})
\end{eqnarray*}
for every $a \in A$ and $h,k \in H$.

\vd

With respect to $\omega$ and $\omega'$, first we observe  that
\[
\omega(h,1_H) = \sum \gamma(h_{(1)}) \gamma(1_H) \gamma'(h_{(2)}) = \gamma(h_{(1)})\gamma'(h_{(2)}) = h \cdot \um ,
\]
and also $\omega(1_H,h) = h \cdot \um ,$ showing that $\omega$ is normalized.  Note, furthermore, that
\begin{equation}\label{absorbs hDotUm}
\sum \omega (h_{(1)} , k) (h_{(2)} \cdot \um) = \sum (h_{(1)} \cdot \um) \omega (h_{(2)} , k) = \omega(h , k ) ,
\end{equation} because $h \cdot \um = (\gamma * \gamma')(h)$ and $ {\bf e} = \gamma * \gamma'$ is central in $\Hom(H,A)$ by Remark~\ref{rem:central}. Therefore:
\begin{eqnarray*}
\sum \omega (h_{(1)} , k) (h_{(2)} \cdot \um)
& = &  \sum (h_{(1)} \cdot \um)\omega (h_{(2)} , k)\\
& = &  \sum \underbrace{\gamma(h_{(1)}) \gamma'(h_{(2)})\gamma(h_{(3)})} \gamma(k_{(1)}) \gamma'(h_{(4)}k_{(2)}) \\
& \overset{\text{(\ref{produtogama})  }}{=}   &  \sum \gamma(h_{(1)}) \gamma(k_{(1)}) \gamma'(h_{(2)}k_{(2)}) = \omega (h , k).
\end{eqnarray*}  Analogously, using  (\ref{produtogamalinha}), one shows that
\begin{equation}\label{omegaLinhaAbsorbs hDotUm}
\sum {\omega}' (h_{(1)} , k) (h_{(2)} \cdot \um) = \sum (h_{(1)} \cdot \um) {\omega }' (h_{(2)} , k) = \omega(h , k ).
\end{equation} For $\omega * \omega'$  we have

\begin{eqnarray*}
(\omega * \omega') (h \ot k)
& = & \sum \underbrace{\gamma(h_{(1)}) \gamma (k_{(1)}) \gamma'(h_{(2)}k_{(2)}) \gamma(h_{(3)}k_{(3)})} \gamma'(k_{(4)}) \gamma'(h_{(4)})   \\
& \overset{\text{(\ref{gammagamma}) }}{=} & \sum \gamma(h_{(1)}) \gamma (k_{(1)}) \gamma'(k_{(2)}) \gamma'(h_{(2)})   =  h \cdot (k \cdot \um).
 \end{eqnarray*}


For the  evaluation of $\omega' * \omega$ we  use (vi) and (vii) of Definition~\ref{cleft} to compute
\begin{eqnarray*}
(\omega' * \omega) (h \ot k)
& = & \sum \gamma(h_{(1)}k_{(1)}) \gamma'(k_{(2)}) \underbrace{ \gamma'(h_{(2)})\gamma(h_{(3)}) } \gamma (k_{(3)}) \gamma'(h_{(4)}k_{(4)})    \\
& = & \sum \gamma(h_{(1)}k_{(1)}) \underbrace { \gamma'(k_{(2)}) \tilde{e}_{h_{(2)}} } \gamma (k_{(3)}) \gamma'(h_{(3)}k_{(4)})    \\
& \overset{\text{ (vi) }}{=} & \sum \underbrace{ \gamma(h_{(1)}k_{(1)})  \tilde{e}_{h_{(2)}k_{(2)}}  }  \gamma'(k_{(3)}) \gamma (k_{(4)}) \gamma'(h_{(3)}k_{(5)})    \\
& \overset{\text{ (\ref{produtogamalinha})  }}{=} & \sum \gamma(h_{(1)}k_{(1)})    \gamma'(k_{(2)}) \gamma (k_{(3)}) \gamma'(h_{(2)}k_{(4)}) \\
&   = & \sum \underbrace { \gamma(h_{(1)}k_{(1)})    \tilde{e}_{ k_{(2)}  } } \gamma'(h_{(2)}k_{(3)})    \overset{\text{ (vi) }}{=}
\sum    {e}_{ h_{(1)}}    \gamma(h_{(2)}k_{(1)})    \gamma'(h_{(3)}k_{(2)})    \\
&= & \sum \gamma(h_{(1)}) \gamma'(h_{(2)}) \gamma (h_{(3)}k_{(1)})\gamma'(h_{(4)}k_{(2)})   =
\sum (h_{(1)} \cdot \um ) (h_{(2)}k \cdot \um ).
\end{eqnarray*}
We use this to obtain   the initial form of the twisting condition given in  (\ref{torcao}) of Definition~\ref{defi:twisted}:
\begin{eqnarray*}
& & \sum (h_{(1)} \cdot (k_{(1)} \cdot a ))\omega(h_{(2)} , k_{(2)}) = \\
& = & \sum \omega(h_{(1)} , k_{(1)})(h_{(2)}k_{(2)} \cdot  a)\omega'(h_{(3)} , k_{(3)})\omega(h_{(4)} , k_{(4)}) \\
& = & \sum \omega(h_{(1)} , k_{(1)})(h_{(2)}k_{(2)} \cdot  a)(h_{(3)} \cdot \um)(h_{(4)} k_{(3)} \cdot \um) \\
& = & \sum \underbrace{\omega(h_{(1)} , k_{(1)})(h_{(2)} \cdot \um)}\underbrace{(h_{(3)}k_{(2)} \cdot  a)(h_{(4)} k_{(3)} \cdot \um)} \\
& = & \sum \omega(h_{(1)} , k_{(1)})(h_{(2)}k_{(2)} \cdot  a).
\end{eqnarray*}  Using again $\omega' * \omega,$  we obtain the similar twisting equality for ${\omega}':$
\begin{equation}\label{OmegaLinhaTorcao}
\sum \omega'(h_{(1)} , k_{(1)}) (h_{(2)} \cdot (k_{(2)} \cdot a))   = \sum(h_{(1)}k_{(1)} \cdot a) \omega'(h_{(2)} , k_{(2)}).
\end{equation}  Indeed, multiplying the above obtained equality
$$ h \cdot (k \cdot a)  = \sum \omega(h_{(1)} , k_{(1)})(h_{(2)}k_{(2)} \cdot  a)\omega'(h_{(3)} , k_{(3)})$$ by ${\omega}'$ on the left, and using  the convolution centrality of ${\bf e}$ and  (\ref{omegaLinhaAbsorbs hDotUm}), we have:
\begin{align*}
 & \sum {\omega}' (h_{(1)} , k_{(1)} ) (h \cdot (k \cdot a))   =\\
&\sum ( {\omega}'  \ast \omega  )  (h_{(1)} , k_{(1)})   (h_{(2)}k_{(2)} \cdot  a)\omega'(h_{(3)} , k_{(3)})=\\
&\sum (h_{(1)} \cdot \um ) \underbrace{  (h_{(2)}k_{(1)} \cdot \um )   (h_{(3)}k_{(2)} \cdot  a) }   \omega'(h_{(4)} , k_{(3)})=\\
&\sum (h_{(1)} \cdot \um )   (h_{(2)}k_{(1)} \cdot  a)    \omega'(h_{(3)} , k_{(2)})=\\
&\sum   (h_{(1)}k_{(1)} \cdot  a)\underbrace{  (h_{(2)} \cdot \um )      \omega'(h_{(3)} , k_{(2)})} =\\
&\sum(h_{(1)}k_{(1)} \cdot a) \omega'(h_{(2)} , k_{(2)}),
\end{align*} as desired.
Note now that by  (v) and (iii) of Definition~\ref{cleft}  we have
\begin{align*}  h \cdot ( k \cdot \um) & =   \sum \gamma (h_{(1)}) e_k {\gamma }'(h_{(2)}) =
\sum  e_{h_{(1)} k }   \gamma (h_{(2)})  {\gamma }'(h_{(3)})  \\
& = \sum (h_{(1)} k \cdot \um ) ( h_{(2)} \cdot \um ) = \sum ( h_{(1)} \cdot \um ) (h_{(2)} k \cdot \um ),
\end{align*} which gives  (iii) of Definition~\ref{symm}.  Next we see that  $\omega$ absorbs  $hk \cdot \um$ on the right:
\begin{align*} \sum \omega (h_{(1)}, k_{(1)}) (h_{(2)} k_{(2)} \cdot \um) & =  \sum \gamma( h_{(1)} ) \gm (k _{(1)})
\underbrace{  {\gm}' (h_{(2)} k_{(2)})  {\gm} (h _{(3)} k_{(3)}) {\gm} '( h _{(4)} k_{(4)})}  \\
&\overset{\text{(\ref{produtogamalinha}) }}{=}  \sum \gamma( h_{(1)} ) \gm (k _{(1)})  {\gm} '( h _{(4)} k_{(4)})  =\omega (h,k),
\end{align*}  showing that  (\ref{cociclo}) holds.  Then using the twisting condition   (\ref{torcao})  we see that
\[
 \sum (h_{(1)} \cdot (k_{(1)} \cdot \um)) \omega (h_{(2)} , k_{(2)}) = \sum \omega(h_{(1)},k_{(1)})(h_{(2)}k_{(2)} \cdot \um) = \omega(h , k).
\] Thus we have that  $\omega(h,k)$ absorbs the elements $h \cdot k \cdot \um,$ $h\cdot \um $ and $hk\cdot \um$ from both sides for any
$h,k \in H.$ In particular, $\omega $ is contained in the ideal $\langle f_1 \ast f_2 \rangle .$

Now, $\omega'$ absorbs $(h \cdot (k \cdot \um))$ on the right, which we see by using (vi) of Definition~\ref{cleft}:
\begin{eqnarray*}
& & \sum \omega'(h_{(1)} , k_{(1)})(h_{(2)} \cdot (k_{(2)} \cdot \um)) = \\
& = & \sum \gamma(h_{(1)}k_{(1)}) \gamma'(k_{(2)}) \gamma'(h_{(2)}) \gamma(h_{(3)}) \gamma(k_{(3)}) \gamma'(k_{(4)}) \gamma'(h_{(4)}) \\
& = & \sum \gamma(h_{(1)}k_{(1)}) \underbrace{ \gamma'(k_{(2)}) \tilde{e}_{h_{(2)}} } \gamma(k_{(3)}) \gamma'(k_{(4)}) \gamma'(h_{(3)}) \\
&\overset{\text{Def.~\ref{cleft}.(vi)}}{=} & \sum \underbrace{ \gamma(h_{(1)}k_{(1)}) \tilde{e}_{h_{(2)} k_{(2)} }  } \underbrace{  \gamma'(k_{(3)})   \gamma(k_{(4)}) \gamma'(k_{(5)}) } \gamma'(h_{(3)})\\
&\overset{\text{(\ref{produtogama}), (\ref{produtogamalinha})  }}{=}  & \sum  \gamma(h_{(1)}k_{(1)})    \gamma'(k_{(2)})    \gamma'(h_{(2)}) = \omega'(h , k),
\end{eqnarray*} By (\ref{omegaLinhaAbsorbs hDotUm}) and the convolution centrality of ${\bf e},$ this implies
$$\sum {\omega}' (h_{(1)}), k_{(1)})  ( h_{(2)} k_{(2)}   \cdot \um ) =   {\omega}' (h , k ) ,$$  and, moreover, it follows using (\ref{OmegaLinhaTorcao}) that
$$\omega '(h,k) = \sum \omega'(h_{(1)} , k_{(1)}) (h_{(2)} \cdot (k_{(2)} \cdot \um ))   = \sum(h_{(1)}k_{(1)} \cdot \um ) \omega'(h_{(2)} , k_{(2)}).$$ Consequently, ${\omega}'$ also absorbs the elements $h \cdot k \cdot \um,$ $h\cdot \um $ and $hk\cdot \um$ from both sides.  In particular,  (\ref{abs.omegalinha})   is satisfied, i.e. ${\omega }'$ belongs to  $\langle f_1 \ast f_2 \rangle .$

We check the cocycle equality  (\ref{9}) for  $\omega ,$ taking into account that ${\gm}' \ast {\gm}$ commutes with each element of $A,$  as follows:
\begin{align*}
& \sum [h_{(1)} \cdot \omega(k_{(1)} , l_{(1)}) ]  \; \omega(h_{(2)} , k_{(2)}  l_{(2)}) \\
&  = \sum \gamma(h_{(1)})    \omega(k_{(1)} , l_{(1)})  [ (\gamma ' \ast \gamma ) (h_{(2)}) ]
  \gamma(k_{(3)} l_{(3)})  \gamma'(h_{(3)} k_{(2)}  l_{(2)} )  \\
& = \sum \underbrace{  \gamma(h_{(1)}) [ (\gamma ' \ast \gamma ) (h_{(2)}) ]    }  \omega(k_{(1)} , l_{(1)})   \gamma(k_{(3)} l_{(3)})  \gamma'(h_{(3)} k_{(2)}  l_{(2)} ) \\
& \overset{\text{(\ref{produtogama}) }}{=}  \sum  \gamma(h_{(1)})  \underbrace{    \; \gamma(k_{(1)} )  \gamma(l_{(1)} ) {\gamma }' (k_{(2)} l_{(2)})    \gamma(k_{(3)} l_{(3)})  }  \gamma'(h_{(2)} k_{(4)}  l_{(4)} ) \\
&  \overset{\text{(\ref{gammagamma}) }}{=}  \sum \underbrace{ \gamma(h_{(1)})     \; \gamma(k_{(1)} )  } \gamma(l_{(1)} )  \gamma'(h_{(2)} k_{(2)}  l_{(2)} ) \\
&   \overset{\text{(\ref{gammagamma}) }}{=}   \sum  [ \gamma(h_{(1)})     \; \gamma(k_{(1)} )   \gamma'(h_{(2)} k_{(2)}  ) ]  [   \gamma (h_{(3)} k_{(3)}   )  \gamma(l_{(1)} )  \gamma'(h_{(4)} k_{(4)}  l_{(2)} ) ]  \\
 & = \sum \omega(h_{(1)} , k_{(1)}) \; \omega(h_{(2)} k_{(2)} ,  l ).
\end{align*} This completes the proof of the fact  that $A = (A, \cdot, \omega, \omega')$ is a symmetric twisted partial $H$-module algebra.

\vd

Finally, we claim that
\begin{eqnarray*}
\Phi: \AH & \rightarrow & B \\
a \# h &\mapsto & a \gamma(h)
\end{eqnarray*}
is an algebra isomorphism, with inverse given by
\begin{eqnarray*}
\Psi: B & \rightarrow & \AH \\
b &\mapsto & \sum b_{(0)} \gamma'(b_{(1)}) \# b_{(2)}
\end{eqnarray*}
In fact, $\Phi$ is an algebra map since it obviously takes unity to unity and
\begin{eqnarray*}
\Phi(a \# h) \Phi(b \# k) & = & a \gamma(h) b \gamma(k) \\
& \overset{\text{(\ref{gammavezesa})}}{=} & \sum a \gamma(h_{(1)})b \gamma'(h_{(2)}) \gamma(h_{(3)}) \gamma(k) \\
& \overset{\text{(\ref{gammagamma})}}{=} & \sum a [\gamma(h_{(1)})b \gamma'(h_{(2)})] [\gamma(h_{(3)}) \gamma(k_{(1)}) \gamma'(h_{(4)}k_{(2)})] \times \\
& & \gamma(h_{(5)}k_{(3)}) \\
&=& \sum a (h_{(1)} \cdot b) \omega(h_{(2)} \ot k_{(1)}))\gamma(h_{(3)}k_{(2)}) \\
&=& \Phi(\sum a (h_{(1)} \cdot b) \omega(h_{(2)} \ot k_{(1)})) \# h_{(3)}k_{(2)}) \\
&=& \Phi((a \# h) \Phi(b \# k)) \\
\end{eqnarray*}

In order to prove that $\Psi = \Phi^{-1}$, first note that $\sum b_{(0)} \gamma'(b_{(1)})$ lies in $A$, because \ref{cleft}.ii implies that
\begin{eqnarray*}
\rho(\sum b_{(0)}\gamma'(b_{(1)})) & = & \sum \rho(b_{(0)})
(\rho \circ \gamma')(b_{(1)}) = \\
= \sum b_{(0)} (\gamma'(b_{(3)})\ot b_{(1)} S(b_{(2)}) & = & \sum b_{(0)} \gamma(b_{(1)}) \ot 1_H.
\end{eqnarray*}

\vu

Now, $\Phi \Psi = Id_B$ is just condition (iv) of definition \ref{cleft}. For the other composition,
 given $a \#h \in \AH$, since $\gamma$ is a comodule morphism and $A = B^{coH}$ it follows that
\begin{eqnarray*}
\Psi (\Phi (a \# h)) & = & \Psi (a \gamma(h))  =  \sum a\gamma(h_{(1)}) \gamma'(h_{(2)}) \# h_{(3)} \\
& = & \sum a (h_{(1)} \cdot \um) \# h_{(2)} = a \# h.
\end{eqnarray*}

\end{proof}




\end{document}